\documentclass[a4paper, 11pt, headings = small, abstract]{scrartcl}
\usepackage[english]{babel}

\usepackage{amscd,amsmath,amsthm,array,bbm,hhline,mathrsfs, enumerate,dsfont, booktabs,fancybox,calc,textcomp,xcolor,graphicx,bbm,xspace,nicefrac,stmaryrd,url,arcs,listings,nameref}
\usepackage[theoremfont]{newpxtext}
\usepackage[vvarbb, upint, bigdelims]{newpxmath}
\usepackage[scr=boondoxupr, scrscaled = 1.05, cal = euler, calscaled =1.05]{mathalpha}
\usepackage[scaled=0.95]{inconsolata}
\DeclareMathAlphabet{\mathsf}{OT1}{\sfdefault}{m}{n}
\SetMathAlphabet{\mathsf}{bold}{OT1}{\sfdefault}{b}{n}
\usepackage[utf8]{inputenc}

\usepackage[shortlabels]{enumitem}
\usepackage{etoolbox}
\usepackage[normalem]{ulem}
\usepackage{mathtools}
\usepackage{geometry}
\usepackage{float}

\usepackage{bm}
\usepackage{tikz}
\usepackage[outdir =./]{epstopdf}
\usepackage[citestyle = numeric-comp,bibstyle = numeric, firstinits=true, natbib = true, backend = bibtex,maxbibnames=99,url=false,doi=false]{biblatex}
\usepackage{authblk}
\numberwithin{equation}{section}

\usepackage{chngcntr}
\counterwithin{figure}{section}

\usepackage{xcolor}
\usepackage[colorlinks=true, allcolors=myteal]{hyperref}
\definecolor{WIMgreen}{RGB}{60 134 132}
\definecolor{red_pers}{RGB}{204 37 41}
\definecolor{UMblue}{RGB}{4 47 86}
\definecolor{myteal}{RGB}{0 123 137}
\definecolor{nd}{RGB}{0 0 0}
\definecolor{dartmouthgreen}{rgb}{0.05, 0.5, 0.06}\definecolor{cobalt}{rgb}{0.0, 0.28, 0.67}\definecolor{coolblack}{rgb}{0.0, 0.18, 0.39}
\definecolor{glaucous}{rgb}{0.38, 0.51, 0.71}\definecolor{hooker\'sgreen}{rgb}{0.0, 0.44, 0.0}\definecolor{lemonchiffon}{rgb}{1.0, 0.98, 0.8}\definecolor{oucrimsonred}{rgb}{0.6, 0.0, 0.0}\definecolor{radicalred}{rgb}{1.0, 0.21, 0.37}\definecolor{raspberry}{rgb}{0.89, 0.04, 0.36}\definecolor{royalazure}{rgb}{0.0, 0.22, 0.66}
\definecolor{dex}{rgb}{0.05,0.5,0.06}

\theoremstyle{plain}
\newtheorem{theorem}{Theorem}[section]
\newtheorem{proposition}[theorem]{Proposition}
\newtheorem{lemma}[theorem]{Lemma}
\newtheorem{corollary}[theorem]{Corollary}

\theoremstyle{definition}
\newtheorem{definition}[theorem]{Definition}
\theoremstyle{assumption}

\theoremstyle{remark}
\newtheorem{remark}[theorem]{Remark}

\SetLabelAlign{center}{\hss#1\hss}

\def\supp{\operatorname{supp}}

\def\TV{\operatorname{TV}}

\def\E{\mathbb{E}}

\def\N{\mathbb{N}}
\def\P{\mathbb{P}}
\def\N{\mathbb{N}}

\def\R{\mathbb{R}}
\def\S{\mathcal{S}}\definecolor{darkred}{rgb}{0,0.6,0}
\def\X{\mathbf{X}}

\def\cR{\mathcal{R}}

\def\cN{\mathcal{N}}

\def\cS{\mathcal{S}}

\def\Var{\mathrm{Var}}

\newcommand{\cA}{\mathscr{A}}
\newcommand{\cB}{\mathcal{B}}

\newcommand{\cE}{\mathcal{E}}

\newcommand{\cG}{\mathcal{G}}

\newcommand{\cL}{\mathcal{L}}

\newcommand{\ep}{\varepsilon}
\newcommand{\cO}{\mathcal{O}}

\newcommand{\Pro}{\mathbb{P}}
\renewcommand{\c}{^{\operatorname{c}}}

\renewcommand{\hat}{\widehat}
\newcommand{\e}{\mathrm{e}}
\renewcommand{\tilde}{\widetilde}%
\renewcommand{\d}{\mathop{}\!\mathrm{d} }

\newcommand{\qv}[1]{\langle#1\rangle}

\newcommand{\1}{\mathbf{1}}

\restylefloat{table}

\newcommand{\cH}{\mathcal{H}}

\newcommand{\vertiii}[1]{{\left\vert\kern-0.25ex\left\vert\kern-0.25ex\left\vert #1
		\right\vert\kern-0.25ex\right\vert\kern-0.25ex\right\vert}}

\def\bbeta{\boldsymbol{\beta}}

\def\supp{\mathrm{supp}}

\def\h{\boldsymbol{h}}

\let\originalleft\left
\let\originalright\right
\renewcommand{\left}{\mathopen{}\mathclose\bgroup\originalleft}
\renewcommand{\right}{\aftergroup\egroup\originalright}

\definecolor{myteal}{RGB}{0 123 137}
\definecolor{cs}{rgb}{0 0 0}
\definecolor{radicalred}{rgb}{1.0, 0.21, 0.37}
\definecolor{dex}{RGB}{0 0 0}
\allowdisplaybreaks
\bibliography{jumpdrift}
\title{Adaptive nonparametric drift estimation for multivariate jump diffusions under $\sup$-norm risk}
\author{Niklas Dexheimer\thanks{Aarhus University, Department of Mathematics, Ny Munkegade 118, 8000 Aarhus C, Denmark.\newline Email: \href{mailto:dexheimer@math.au.dk}{dexheimer@math.au.dk}
	\newline
	The author gratefully acknowledges financial support of Sapere Aude: DFF-Starting Grant 0165-00061B “Learning diffusion dynamics and strategies for optimal control.”}
}
\begin{document}
\maketitle
\begin{abstract}
We investigate nonparametric drift estimation for multidimensional jump diffusions based on continuous observations. The results are derived under anisotropic smoothness assumptions and the estimators' performance is measured in terms of the $\sup$-norm loss. We present two different Nadaraya--Watson type estimators, which are both shown to achieve the classical nonparametric rate of convergence under varying assumptions on the jump measure. Fully data-driven versions of both estimators are also introduced and shown to attain the same rate of convergence. The results rely on novel uniform moment bounds for empirical processes associated to the investigated jump diffusion, which are of independent interest.
\end{abstract}
\section{Introduction}\label{sec: intro}
Diffusion processes are one of the most fundamental classes of stochastic processes, with their applications ranging from physics and chemistry up to finance or life sciences. Thus a plethora of research has been devoted to their statistical analysis. However, much of this has been focussed on the case of reversible diffusions. \textcolor{dex}{ It is known by Kolmogorov's characterisation, that a stochastic process $\X=(X_t)_{t\geq0}$ satisfying a stochastic differential equation of the form
\[\d X_t= b(X_t)\d t+\d W_t,\quad X_0=\xi, t\geq0, \]
is reversible if, and only if, $b=-\nabla V$ for some suitable potential function $V\colon \R^d\to\R.$ Here $(W_t)_{t\geq0}$ is a $d$-dimensional standard Brownian motion and the random vector $\xi\in\R^d$ is assumed to be independent of $(W_t)_{t\geq0}$.} As the generator of $\X$ is then self-adjoint, this allows the usage of functional inequalities, such as Poincaré or Nash inequalities, for the statistical analysis of $\X$ (see e.g.\ \cite{strauch18,dalrei07,str15}). Apart from the obvious downside of restricting the drift function to be of gradient form, this framework also does not permit the inclusion of jump structures, i.e.\ it does not allow the study of processes $\X=(X_t)_{t\geq0}$ of the form
\begin{equation}\label{eq: jumpdiff intro}
	X_t=\xi+\int_0^t b(X_s)\d s+\int_0^t\sigma(X_s)\d W_s+\int_0^t \int_{\R^d} \gamma(X_{s-})z \tilde{N}(\d s, \d z),\quad t\geq0.
\end{equation}
Here $b\colon \R^d\to\R^d,\sigma\colon\R^d\to\R^{d\times d}, \gamma\colon\R^d\to\R^{d\times d},$  $(W_t)_{t\geq0}$ is as above and $\tilde{N}$ is a compensated Poisson random measure on $[0,\infty)\times \R^d\backslash\{0\}$ with intensity measure $\pi(\d s,\d z)=\d s\otimes \nu(\d z),$ where $\nu$ is a Lévy measure, and $\xi\in\R^d$ is assumed to be independent of $N$ and $(W_t)_{t\geq0}$. 

The focus of this work lies on nonparametric estimation of the drift function $b$ for jump diffusions of the form \eqref{eq: jumpdiff intro}, based on a continuous record of observations. We impose anisotropic Hölder smoothness assumptions and measure the estimation error in the $\sup$-norm, which compared to the $L^2$-risk is generally more demanding as it requires a precise knowledge about the process's concentration behaviour. In order to investigate the arising challenges through the diffusion's discontinuity, we firstly investigate the same Nadaraya--Watson type estimator as employed in the continuous setting in \cite{aeck22}, and show that it attains the classical nonparametric rate of convergence as soon as $\nu$ admits an exponential moment. Building on this we also introduce a more delicate, truncated estimator, which achieves this result as soon as the Lévy-measure admits a third moment.
Furthermore adaptive, i.e.\ fully data-driven, variants of both estimators are introduced and it is proved that they also achieve the classical nonparametric rate of convergence, without requiring any previous knowledge on the drift function. The essential ingredient for all of these results are novel uniform moment bounds for empirical processes associated with the investigated jump diffusion, which are proved through Talagrand's generic chaining device. 

As mentioned above, a similar framework has been investigated for nonreversible continuous diffusion processes in the recent reference \cite{aeck22}. Therein an adaptive estimator is presented, which attains the classical nonparametric rate of convergence, thus achieving minimax optimality. The results also rely on uniform moment bounds, which are proved through a similar technique as introduced in \cite{kindiff} for drift estimation for a class of hypoelliptic diffusions, so-called stochastic damping Hamiltonian systems. As this technique strongly relies on the continuity of the investigated processes, generalising this approach towards jump diffusions of the form \eqref{eq: jumpdiff intro} is nontrivial. For a more detailed discussion see Section \ref{sec: unif mom}. 
Apart from this there exist many other works focussed on nonparametric drift estimation for continuous diffusions (see e.g.\ \cite{str15,dalrei07,dala05}) or parametric drift estimation for jump diffusions (see e.g.\cite{gloteraos18,onlinedrift,dexheimer2022lasso,maiou}), however the literature on nonparametric drift estimation for jump diffusions is rather scarce (for an overview see e.g.\ \cite{park21,schmisser14}). In particular, to the best of our knowledge, no results exist in the multidimensional case or on the $\sup$-norm risk. The most closely related work to ours is given by \cite{schmisser14}, where nonparametric drift estimation for jump diffusions based on discrete observations is investigated. The estimator is obtained by minimising a certain contrast function and also an adaptive version is presented. It is shown that the classical nonparametric rate of convergence is achieved by the estimators as soon as the distance between the observations is small enough. However, only the scalar setting is investigated and the aforementioned result on the rate of convergence is shown for the less involved empirical squared risk. Additionally, due to the different observation structure, the assumptions are more restrictive than in our work, as the reference requires symmetry of the Lévy-measure $\nu$ or the jump coefficient $\gamma$ to be constant, as well as a fourth moment of $\nu$.

The structure of this paper is as follows. In Section \ref{sec: ass} we introduce the mathematical framework of this work including all major assumptions on the investigated jump diffusion $\X$ together with some notation. In Section \ref{sec: unif mom} we then prove the needed uniform moment bounds for empirical processes associated with $\X$ under differing assumptions on the Lévy measure $\nu$. In the following Section \ref{sec: drift est} these results are then applied to bounding the $\sup$-norm risk of two variations of Nadaraya--Watson type estimator for the drift function $b$ under anisotropic Hölder smoothness assumptions on $b$. Section \ref{sec: adap} then contains adaptive extensions of the estimators introduced in the previous segment, which are shown to achieve the same rate of convergence, even if the drift's smoothness is unknown. Lastly, the results of this paper are discussed in Section \ref{sec: dis}.
\section{Preliminaries}\label{sec: ass}
\textcolor{dex}{In this section we introduce the required assumptions for our statistical analysis.}
\subsection{Assumptions}
Throughout the whole paper, we assume that the stochastic integral equation \eqref{eq: jumpdiff intro} admits a unique, non-explosive weak solution, which we denote by $\X=(X_t)_{t\geq0}$. \textcolor{dex}{Additionally we assume that $(\X,(\Pro^x)_{x\in\R^d}))$ is a Borel right Markov process.} The following framework of assumptions will be referred to as ($\mathscr{A}$) for the rest of the paper. 
\begin{enumerate}[leftmargin=*,label=($\mathscr{A}$\arabic*),ref=($\mathscr{A}$\arabic*)]\label{framework}
	\item \label{ass: inv} For any $t>0,$ exists a measurable function $p_t\colon\R^d\times \R^d\mapsto\R_+$, referred to as transition density, such that
	\[P_t(x,B)=\int_B p_t(x,y)\d y, \quad B\in \cB(\R^d),x\in\R^d, \]
	i.e.\ the marginal laws of $\X$ are absolutely continuous. Additionally we assume $\X$ to admit a unique absolutely continuous invariant distribution $\mu$ with density $\rho\colon \R^d\mapsto\R_+$, i.e.\ for any $t>0, B\in\cB(\R^d),$
	\[\Pro^\mu (X_t\in B)\coloneq \int P_t(x, B)\mu(\d x)=\int_{\R^d}\int_B p_t(x,y)\d y\rho(x)\d x=\int_B \rho(x)\d x=\mu(B). \]
	\item 
 There exists a constant $c_0>0,$ such that the following transition density bound holds true:
\[
		\forall t\in (0,1]: \quad \sup_{x,y \in \R^d} p_t(x,y) \leq c_0 t^{-d \slash 2}.\]
	\label{ass: heat}
	\item $\X$ started in the invariant measure $\mu$ is exponentially $\beta$-mixing, i.e.\  \label{cond2:mixing} there exist constants $c_\kappa,\kappa>0$ such that
	\[\int \Vert P_t(x,\cdot)-\mu(\cdot)\Vert_{\TV}\mu(\d x)\leq c_\kappa \exp(-\kappa t), \quad t\geq0, \]
	where $\Vert \cdot\Vert_{\TV}$ denotes the total variation norm.
	\label{ass: mixing}
\end{enumerate}
These assumptions on the investigated process are of course rather abstract, but give an overview of which properties are actually needed for our analysis and also allow our statistical results to be adapted to new probabilistic results. In practice however, one may be more interested in precise conditions on the coefficients $b,\sigma,\gamma$ and the Lévy-measure $\nu,$ than the theoretical framework above. Therefore we introduce the following set of assumptions.
\begin{enumerate}[leftmargin=*,label=($\mathscr{J}$\arabic*),ref=($\mathscr{J}$\arabic*)]
	\item  \label{Lipschitz assumptions}The functions $b,\gamma,\sigma$ are globally Lipschitz continuous, $b$ and $\gamma$ are bounded, and, for $\mathbb{I}_{d\times d}$ denoting the $d\times d$-identity matrix, there exists a constant $c\geq1$ such that \[\forall x\in\R^d\colon c^{-1}\mathbb{I}_{d\times d}\leq \sigma(x)\sigma^\top(x)\leq c\mathbb{I}_{d\times d},\]
	where the ordering is in the sense of Loewner for positive semi-definite matrices.
	\item \label{Kappa Assumptions}$\nu$ is absolutely continuous with respect to the Lebesgue measure and, for an $\alpha\in(0,2)$, \[(x,z) \mapsto \Vert \gamma(x)z\Vert^{d+\alpha}\nu(z)\]
	is bounded and measurable, where, by abuse of notation, we denoted the density of $\nu$ also by $\nu$. Furthermore, 
	if $\alpha=1$, \[\int_{r<\Vert\gamma(x) z\Vert\leq R} \gamma(x)z\,\nu(\mathrm{d}z)=0,\quad\text{ for any } 0<r<R<\infty,\ x\in\R^d.\] 
	\item \label{Ergodicity Assumptions} There exist $\eta_0,\eta_1,\eta_2>0$ such that
	\[\qv{x,b(x)}\leq -\eta_0\Vert x\Vert, \quad \forall x:\Vert x\Vert\geq \eta_1,\quad\text{ and }\quad
	\int_{ \R^d}\Vert z\Vert^2 \e^{\eta_2\Vert z\Vert}\nu(\mathrm{d}z)<\infty.\]
\end{enumerate}
The above assumptions are the same as investigated in Example 4.3 of \cite{dexheimeraihp}, where it is also shown that $\mathscr{A}$ holds as soon as \ref{Lipschitz assumptions},\ref{Kappa Assumptions} and \ref{Ergodicity Assumptions} are fulfilled.
\subsection{Additional Assumptions and Notation}\label{subsec: not}
Throughout the whole paper we let $X_0=\xi\sim\mu,$ i.e.\ $\X$ is stationary and denote $\Pro=\Pro^\mu,\E=\E^\mu$. For a function $f\in L^1(\mu)$ we denote $\mu(f)\coloneq\int f\d \mu$ and we introduce the following notation for the diffusion coefficient $a\colon \R^d\to\R^{d\times d}, x\mapsto\sigma(x) \sigma^\top(x)=a(x).$ Furthermore for $n\in\N$ we denote $[n]\coloneq\{1,\ldots,n\}$. Additionally, we assume the drift $b$ to be locally bounded, and $a,\gamma$ to be globally bounded. Moreover for $x\in\R^d,\ep>0,$ we denote $B(x,\ep)\coloneq\{y\in\R^d\colon \Vert x-y\Vert<\ep \}$ and for a class of functions $\cG$ mapping from $\R^d$ to $\R$ and a function $f\colon\R^d\to\R$ we set $\cG f\coloneq \{fg\colon g\in\cG \}.$ Lastly, we define the $\sup$-norm risk of an estimator $\hat{f}$ of a function $f$ on a domain $D\subset \R^d$ as
\[\cR(\hat{f},f;D)\coloneq \E[\Vert f-\hat{f}\Vert^p_{L^\infty(D)} ]^{1/p}, \quad p\geq 1, \]
where $\Vert \cdot\Vert_{L^\infty(D)}$ denotes the restriction of the $\sup$-norm to $D$. In order to improve the paper's readability all proofs have been deferred to the appendix. Furthermore $c$ always denotes a strictly positive constant, whose value may change from line to line. In contrast, specific constants are denoted with an additional subscript.
\section{Uniform moment bounds}\label{sec: unif mom}
As we intend to bound the $\sup$-norm risk of a Nadaraya--Watson type drift estimator, we are in need of uniform moment bounds over certain classes of functions for empirical processes associated to the investigated jump diffusion $\X$. More precisely, for a bounded, measurable function $g$ on $\R^d$ and $j\in[d]$ we define
\begin{align*}
	\mathbb{I}^j_T(g)&\coloneq \frac{1}{\sqrt{T}}\int_0^T g(X_{s-}) \d X^j_s
	\\&= \mathbb{H}_T(gb^j)+\mathbb{M}^j_T(g)+\mathbb{J}^j_T(g),
\end{align*}
where
\begin{align}
	\mathbb{H}_T(g)&= \frac{1}{\sqrt{T}}\int_0^T g(X_s) \d s,\notag
	\\
	\mathbb{M}^j_T(g)&=\frac{1}{\sqrt{T}}\int_0^T g(X_s)\sum_{k=1}^{d}\sigma_{jk}(X_s) \d W^k_s,\label{def: emp proc}
	\\
	\mathbb{J}^j_T(g)&=\frac{1}{\sqrt{T}}\int_0^T\int_{\R^d} g(X_{s-})\gamma^j(X_{s-})z \tilde{N}(\d s,\d z),\notag
\end{align}
with $\gamma^j$ denoting the $j$-th row of $\gamma,$ and aim at finding a suitable uniform moment bound of the centred version of this functional over a countable class $\cG$ of such functions, i.e.\ for $j\in[d],p\geq1,$ we want to bound
\[\E\Big[\sup_{g\in\cG}\Vert \mathbb{I}^j_T(g)-\sqrt{T}\mu(gb^j)\Vert^p\Big]^{1/p}.\]
By the Minkowski inequality it obviously suffices to find uniform moments bounds for the three summands defined in \eqref{def: emp proc}. A crucial ingredient for doing this will be the generalisation of Talagrand's generic chaining device contained in \cite{dirksen15}, which entails the needed bounds as soon as the investigated stochastic process fulfills a Bernstein type concentration inequality. Regarding $\mathbb{H}_T$ we can employ the results of \cite{dexheimeraihp}, where, as in \cite{viennet}, Berbee's coupling method is applied together with the exponential $\beta$-mixing property to obtain the needed concentration behaviour through the standard Bernstein inequality for independent, identically distributed random variables. Using this result together with Bernstein's inequality for continuous martingales (see, e.g., p.\ 153 in \cite{revuzyor1999}) uniform moment bounds for functionals akin to $\mathbb{M}^j_T$ were shown in \cite{kindiff} and can be derived similarly in the setting of this paper. Before we present this result, recall that for $\ep>0$ the covering number $\cN(\ep,\cG,d)$ of a set $\cG$ with respect to a semi-metric $d$ denotes the smallest number of open balls of $d$-radius $\ep$ needed to cover $\cG$. Furthermore for two measurable, bounded functions $f,g\colon \R^d\to\R$ we define
\[d_{\infty}(f,g)\coloneq \Vert f-g\Vert_\infty,\quad d^p_{L^p(\mu)}(f,g)\coloneq \mu(\vert f-g\vert^p),\quad p\geq1. \] Our uniform moment bounds for $\mathbb{M}^j_T$ then read as follows.
\begin{proposition}\label{prop: cont moment bound}
	Grant assumptions \ref{ass: inv}, \ref{ass: mixing} and let $\cG$ be a countable class of bounded, measurable functions mapping from $\R^d \to\R$. Then for large enough $T,$ it holds for any $p\in[1,\infty)$ and $j\in[d]$ 
	\begin{align*}
		&\E\left[\sup_{g\in\mathcal{G}}\vert \mathbb{M}^j_T\vert^p\right]^{1/p}
		\\
		&\leq C_1\Bigg(\int_0^\infty \log\left(\mathcal{N}(u,\mathcal{G},24\sqrt{\Vert a_{jj}\Vert_{\infty}}((\kappa T)^{-1/4}d_\infty+(\kappa T)^{-1/8}d_{L^4(\mu)})) \right) \d u
		\\
		&\quad +\int_0^\infty\sqrt{\log\left(\mathcal{N}(u,\mathcal{G},\sqrt{6\Vert a_{jj}\Vert_{\infty}} (128^{1/4}(\kappa T)^{-1/8}d_{L^4(\mu)}+d_{L^2(\mu)}))\right) }\d u\Bigg)
		\\&\quad+ 24p\sqrt{\Vert a_{jj}\Vert_\infty}\tilde{c}_1\sup_{g\in\mathcal{G}}\left((\kappa T)^{-1/4}\Vert g\Vert_\infty+(\kappa T)^{-1/8}\Vert g\Vert_{L^4(\mu)}\right)
		\\&\quad+\sqrt{6\Vert a_{jj}\Vert_\infty p}\tilde{c}_2\sup_{g\in\mathcal{G}}\left((128)^{1/4}(\kappa T)^{-1/8}\Vert g\Vert_{L^4(\mu)}+\Vert g\Vert_{L^2(\mu)}\right),
	\end{align*}
 where
\[\tilde{c}_1=4\e^{1/(2\e)}(\sqrt{8\pi}\e^{1/(12p)})^{1/p}\e^{-1},\quad \tilde{c}_2=4\e^{1/(2\e)}(\sqrt{2\pi}\e^{1/(6p)})^{1/p}\e^{-1/2}.  \]
\end{proposition}
For the analysis of the jump part $\mathbb{J}^j_T$ Bernstein's inequality for continuous martingales is naturally not available. For verifying the required concentration behaviour we therefore employ the exponential martingale inequality given in Theorem 2.2 of \cite{applebaum09paper}, which adjusted to the setting of this paper gives for any $T,\alpha,\beta>0$ that
\begin{equation}
\Pro\left(\int_0^T\int \gamma(X_{s-})z\tilde{N}(\d s,\d z)- \frac{1}{\alpha}\int_0^T\int \exp(\alpha \gamma(X_{s-}z))-1-\alpha \gamma(X_{s-})z\nu(\d z)\d s>\beta \right)\leq \exp(-\alpha\beta). \label{eq: exp martingale ineq}
\end{equation}
Clearly, the usefulness of this inequality strongly depends on the Lévy measure $\nu,$ which is why we introduce the following exponential moment assumption.
\begin{enumerate}[leftmargin=*,label=($\mathcal{N}_{\exp}$),ref=($\mathcal{N}_{\exp}$)]
\item There exists a constant $c_{1,\nu}>0,$ such that\label{ass: exp mom}
\[\int_{\Vert x\Vert\geq1} \exp(c_{1,\nu}\Vert x\Vert)\nu(\d x)<\infty. \]
\end{enumerate}
Under the above assumption we are able to prove bounds similar to the ones for the continuous martingale part observed in Proposition \ref{prop: cont moment bound}. 
\begin{proposition}\label{prop: jump moment bound}
	Let everything be given as in Proposition \ref{prop: cont moment bound} and assume \ref{ass: exp mom}. Then there exists a universal constant $C_1>0,$ such that for large enough $T,$ it holds for any $p\in[1,\infty)$ and $j\in[d]$
	\begin{align*}
		&\E\Big[\sup_{g\in\mathcal{G}}\vert \mathbb{J}^j_T(g)\vert^p\Big]^{1/p}
		\\ &\leq C_1\Bigg(\int_0^\infty \log\left(\mathcal{N}(u,\mathcal{G},3T^{-1/8}d_{L^4(\mu)}+3(64c_1\kappa^{-1/2}+1)T^{-1/4}d_\infty) \right) \d u
		\\
		&\quad +\int_0^\infty\sqrt{\log\left(\mathcal{N}(u,\mathcal{G},\sqrt{3}(1+c_1)d_{L^2(\mu)}+\sqrt{3072}c_1\kappa^{-1/4}T^{-1/8}d_{L^4(\mu)}) \right) }\d u\Bigg)
		\\
		&\quad+3p\tilde{c}_1\sup_{g\in\mathcal{G}}\left(T^{-1/8}\Vert g\Vert_{L^4(\mu)}+(64c_1\kappa^{-1/2}+1)T^{-1/4}\Vert g\Vert_\infty  \right)
		\\&\quad+\sqrt{3p}\tilde{c}_2\sup_{g\in\mathcal{G}}\left((1+c_1)\Vert g\Vert_{L^2(\mu)}+32c_1\kappa^{-1/4}T^{-1/8}\Vert g\Vert_{L^4(\mu)}\right),
	\end{align*}
	where $c_1$ is defined in \eqref{eq:  def c_1}. 
\end{proposition}
One might be tempted to think that the exponential moment assumption \ref{ass: exp mom} is necessary for $\X$ to be exponentially $\beta$-mixing. However, in \cite{mas04} it was shown that under a certain drift condition Lévy-driven Ornstein--Uhlenbeck processes fulfill \ref{ass: mixing} as soon as the Lévy measure $\nu$ admits a polynomial moment. Thus, we also investigate the setting in which only the following assumption holds true. 
\begin{enumerate}[leftmargin=*,label=($\mathcal{N}_{\mathrm{tr}}$),ref=($\mathcal{N}_{\mathrm{tr}}$)]
	\item The Lévy measure $\nu$ admits a third moment, i.e. \label{ass: th mom}
	\[\nu_3\coloneq \int \Vert x\Vert^3\nu(\d z)<\infty. \]
	Additionally there exist $\gamma_{\min}>0,$ such that $\gamma_{\min} $ is a global lower bound for the smallest singular value of $\gamma$.
\end{enumerate}
The assumption on the smallest singular value of the jump coefficient $\gamma$ is merely a technical condition, which implies that any jump in the noise also leads to a jump of $\X$.
While the stronger assumption \ref{ass: exp mom} allowed us to find uniform moment bounds for $\mathbb{I}^j_T,$ we can still show similar results under \ref{ass: th mom} for a truncated version of $\mathbb{I}^j_T$.  More precisely, for $T,\delta>0,j\in[d]$ we define
\[X^{j,\delta}_T=X^j_T-\sum_{0\leq s\leq T} \Delta X^j_s\1_{(\delta,\infty)}(\Vert \Delta X_s\Vert),\quad \textrm{where}\quad \Delta X^j_s\coloneq X^j_{s}-X^j_{s-},\]
and 
\[\mathbb{I}^{j,\delta}_T(g)\coloneq \frac{1}{\sqrt{T}}\int_0^Tg(X_{s-})\d X^{j,\delta}_s. \]
Then the following decomposition for the centred moment of $\mathbb{I}^{j,\delta}_T$ holds true.
\begin{lemma}\label{prop: trunc decomp}
	Let $\mathcal{G}$ be as in Proposition \ref{prop: cont moment bound} and assume \ref{ass: th mom}. Then it holds for any $j\in[d],T>0,p\geq1,\delta>0,$
	\begin{align*}
		&\E\left[\sup_{g\in\mathcal{G}}\left\vert\mathbb{I}^{j,\delta}_T(g)-\sqrt{T}\mu(gb)\right\vert^p\right]^{1/p}
		\\
		&\leq \E\left[\sup_{g\in\mathcal{G}} \vert \mathbb{H}_T(gb^j)-\sqrt{T}\mu(gb^j)\vert^p\right]^{1/p}+\E\left[\sup_{g\in\mathcal{G}}\vert \mathbb{M}^j_T(g)\vert^p\right]^{1/p}+\E\left[\sup_{g\in\mathcal{G}}\vert \mathbb{J}^{j,\delta,1}_T(g)\vert^p\right]^{1/p}
		\\
		&\quad+\frac{\Vert \gamma\Vert_\infty\delta}{\gamma_{\min}}\E\left[\sup_{g\in\mathcal{G}}\vert \mathbb{J}^{\delta,2}_T(g)\vert^p\right]^{1/p}+\frac{c_2}{\delta^2}\E\left[\sup_{g\in\mathcal{G}} \vert \mathbb{H}_T(\vert g\vert)-\sqrt{T}\mu(\vert g\vert)\vert^p\right]^{1/p}+\frac{c_2\sqrt{T}}{\delta^2}\sup_{g\in\mathcal{G}}\mu(\vert g\vert),
	\end{align*}
	where
	\begin{align*}
		\mathbb{J}^{j,\delta,1}_T(g)&\coloneq \frac{1}{\sqrt{T}}\int_0^T\int_{B(0,\delta/\gamma_{\min})} g(X_{s-})\gamma^j(X_{s-})z \tilde{N}(\d s,\d z),
		\\
		\mathbb{J}^{\delta,2}_T(g)&\coloneq \frac{1}{\sqrt{T}}\int_0^T\int_{\delta/\Vert \gamma\Vert_\infty<\Vert z\Vert\leq \delta/\gamma_{\min}} \vert g(X_{s-})\vert \tilde{N}(\d s,\d z),
		\\
		c_2&\coloneq \nu_3\Vert \gamma\Vert_\infty\left(\frac{\Vert \gamma\Vert_\infty^3}{\gamma_{\min}}+\gamma_{\min}^2\right),\quad \textrm{with}\quad \nu_3=\int \Vert z\Vert ^3\nu(\d z).
	\end{align*}
\end{lemma}
As we can again employ the results of \cite{dexheimeraihp} and Proposition \ref{prop: cont moment bound}, it now suffices to find suitable bounds for  $\mathbb{J}^{j,\delta,1}_T,\mathbb{J}^{\delta,2}_T$ in order to obtain the desired uniform moment bounds for the truncated empirical process. These are given in the following Proposition.
\begin{proposition}\label{prop: trunc chain}
	Let everything be given as in Proposition \ref{prop: cont moment bound}, assume \ref{ass: th mom} and let $0<\delta\leq \alpha T^{1/4}$, for a fixed constant $\alpha>0$. Then for large enough $T$ it holds for any $j\in[d],p\geq 1,$
	\begin{align*}
		&\E\left[\sup_{g\in\mathcal{G}}\vert \mathbb{J}^{j,\delta,1}_T\vert^p\right]^{1/p}
		\\
		&\leq C_1\Bigg(\int_0^\infty \log\left(\mathcal{N}\left(u,\mathcal{G},3\left(T^{-1/8}d_{L^4(\mu)}+T^{-1/4}\left(1+64\cE_\alpha\kappa^{-1/2} \right)d_\infty\right)\right) \right) \d u
		\\
		&\quad +\int_0^\infty\sqrt{\log\left(\mathcal{N}\left(u,\mathcal{G},\sqrt{3}\left((1+\cE_\alpha)d_{L^2(\mu)}+32\cE_\alpha T^{-1/8}\kappa^{-1/4}d_{L^4(\mu)} \right)\right)\right) }\d u\Bigg)
		\\
		&\quad+\sup_{g\in\mathcal{G}}\Bigg(3p\tilde{c}_1\left(T^{-1/8}\Vert g\Vert_{L^4(\mu)}+T^{-1/4}\left(1+64\cE_\alpha\kappa^{-1/2} \right)\Vert g\Vert_{\infty}\right)
		\\&\quad+\sqrt{3p}\tilde{c}_2\left((1+\cE_\alpha)\Vert g\Vert_{L^2(\mu)}+32\cE_\alpha T^{-1/8}\kappa^{-1/4}\Vert g\Vert_{L^4(\mu)} \right) \Bigg),
	\end{align*}
	and
	\begin{align*}
		&\E\left[\sup_{g\in\mathcal{G}}\vert \mathbb{J}^{\delta,2}_T\vert^p\right]^{1/p}
		\\
		&\leq C_1\Bigg(\int_0^\infty \log\left(\mathcal{N}\left(u,\mathcal{G},3\delta^{-1}\left(T^{-1/8}d_{L^4(\mu)}+T^{-1/4}d_\infty\left(1+32\exp(\alpha)\nu_2\Vert \gamma\Vert_\infty^2\kappa^{-1/2}\right) \right)\right) \right) \d u
		\\
		&\quad +\int_0^\infty\sqrt{\log\left(\mathcal{N}\left(u,\mathcal{G},\sqrt{3}\delta^{-1}\left(\left(1+\frac{\nu_2\Vert \gamma\Vert_\infty^2}{2}\exp(\alpha)\right)d_{L^2(\mu)}+16\exp(\alpha)\nu_2\Vert \gamma\Vert_\infty^2\kappa^{-1/4}T^{-1/8}d_{L^4(\mu)}\right)\right)\right) }\d u\Bigg)
		\\
		&\quad+\delta^{-1}\sup_{g\in\mathcal{G}}\Bigg(3p\tilde{c}_1\left(T^{-1/8}\Vert g\Vert_{L^4(\mu)}+T^{-1/4}\Vert g\Vert_\infty\left(1+32\exp(\alpha)\nu_2\Vert \gamma\Vert_\infty^2\kappa^{-1/2}\right) \right)
		\\&\quad+\sqrt{3p}\tilde{c}_2\left(\left(1+\frac{\nu_2\Vert \gamma\Vert_\infty^2}{2}\exp(\alpha)\right)\Vert g\Vert_{L^2(\mu)}+16\exp(\alpha)\nu_2\Vert \gamma\Vert_\infty^2\kappa^{-1/4}T^{-1/8}\Vert g\Vert_{L^4(\mu)}\right) \Bigg),
	\end{align*}
	where
	\[ \cE_\alpha\coloneq\Vert \gamma\Vert_\infty^2\frac{\nu_2\exp(\alpha\Vert \gamma\Vert_\infty \sigma^{-1}_d)}{2}, \quad \nu_2\coloneq \int\Vert z\Vert^2\nu(\d z).\]
\end{proposition}
Comparable uniform moment bounds for jump processes as in Propositions \ref{prop: jump moment bound} and \ref{prop: trunc chain} have been shown in \cite{wang19} and \cite{vdgeer95}. However, for this \cite{wang19} assumes the jumps to be bounded (see Theorem 1.2 in \cite{wang19}) and \cite{vdgeer95} assumes the jumps to be bounded from below (See Theorem 3.1 in \cite{vdgeer95}). To the best of our knowledge no results under an exponential moment assumption as in Proposition \ref{prop: jump moment bound} were available so far.
 Lastly, denoting $\supp(\cG)\coloneq\bigcup_{g\in\cG}\supp(g),$ one can actually achieve the same results as stated in this section under the less restrictive assumption, that $a$ and $\gamma$ are bounded on $\supp(\cG)$. However, as this does not offer much further insight, we stick to the global boundedness assumptions stated in Section \ref{subsec: not}.

\section{Drift estimation} \label{sec: drift est}
In this section we employ the powerful results of Section \ref{sec: unif mom} to investigate the statistical aim of this work, nonparametric drift estimation for jump diffusions. For the rest of this section $D$ will denote an open, bounded subset of $\R^d$. Furthermore we let $k\colon \R^d\to\R$ denote a symmetric, Lipschitz continuous kernel function with $\supp(k)= [-1/2,1/2]$ of order $\mathsf{k}\in\N,$ i.e.
\[\int k(x)\d x=1,\quad \textrm{and }\forall i\in[\mathsf{k}]: \int k(x)x^i\d x=0. \]
With this we define the following for any multi-index bandwidth $\h=\h(T)=(h_1(T),\ldots,h_d(T))\in(0,1)^d, x\in D,$
\begin{align*}
 K(x)\coloneq \prod_{i=1}^d k(x_i),\quad 	K_{\h}(x)&\coloneq \prod_{i=1}^{d}h_i^{-1} k(x_i/h_i), 
\quad V(\h)\coloneq \prod_{i=1}^d h_i.
\end{align*}
 Furthermore for any $T>0$ we introduce the following class of multi-index bandwidths 
\[ \mathsf{H}_T\coloneq \left\{\h\in(0,1)^d\colon h_i\leq \log(1+T)^{-1}\, \forall i\in[d]\, \mathrm{and}\,  V(\h)\geq T^{-1/2}\log\left(\sum_{i=1}^dh_i^{-1}\right)^4  \right\},\] and the anisotropic Hölder class of functions on $D$.
\begin{definition}\label{def: hölder}
Given $\bbeta=(\beta_1,\ldots,\beta_d), \mathbf{\cL}=(\cL_1,\ldots,\cL_d)\in(0,\infty)^d,$ the anisotropic Hölder class of functions $\cH_D(\bbeta,\mathbf{\cL})$ on $D$ is the set of all functions $f\colon D\to \R$ fulfilling for all $i\in[d],$
\begin{align*}
\Vert D^k_i f\Vert_\infty&\leq \cL_i, \quad \forall k\in[\llfloor \beta_i\rrfloor],
\\
\Vert D^{\llfloor \beta_i\rrfloor}_i f(\cdot +t\mathsf{e}_i)- D^{\llfloor \beta_i\rrfloor}_i f(\cdot)\Vert_\infty &\leq \cL_i\vert t\vert^{\beta_i-\llfloor \beta_i\rrfloor}, \quad \forall t\in\R.
\end{align*}
Here $D^k_i f$ denotes the $k$-th order partial derivative of $f$ with respect to the $i$-th component, $\llfloor \beta\rrfloor$ the largest integer strictly smaller than $\beta$ and $\mathsf{e}_1,\ldots, \mathsf{e}_d$ the canonical basis of $\R^d$. Furthermore we let
\[\bar{\bbeta}\coloneq \frac{d}{\sum_{i=1}^{d}\beta_i^{-1}},\]
denote the harmonic mean of $\bbeta$.
\end{definition}
In order to estimate the drift coefficient $b$ we firstly introduce the following two auxiliary estimators, defined for any $x\in D, \h\in(0,1)^d,j\in[d], \delta>0$ as
\begin{equation}\label{def: aux est}
\bar{b}^{(j)}_{\h,T}(x)\coloneq \frac 1 T \int_0^TK_{\h}(x-X_{s-})\d X^j_s,\quad \bar{b}^{(j),\delta}_{\h,T}(x)\coloneq \frac 1 T \int_0^TK_{\h}(x-X_{s-})\d X^{j,\delta}_s. 
\end{equation}
Note that $\bar{b}^{(j)}_{\h,T}$ is the same estimator as employed for continuous diffusions in \cite{aeck22}, whereas $\bar{b}^{(j),\delta}_{\h,T}$ has been adjusted to the given discountinuous setting.
Recalling the classical decomposition of the $\sup$-norm risk into bias and stochastic error 
\[\mathcal{R}^{(p)}_\infty(b^j\rho,\bar{b}^{(j)}_{\h,T};D)\leq	\mathcal{B}_{b^j\rho}(\h)+\E\left[\sup_{x\in D} \vert \bar{b}^{(j)}_{\h,T}-\mu(b^j\rho)\vert  \right], \]
where the bias is given as
\[\mathcal{B}_{b^j\rho}(\h)\coloneq \sup_{x\in D}\vert (b^j\rho - (b^j\rho)\ast K_{\h})(x)\vert,  \]
we see that the uniform moment bounds of Section \ref{sec: unif mom} can be applied to the function class
\[\mathcal{G}_{\h}\coloneq \left\{\prod_{i=1}^d k((x_i-\cdot)/h_i)\colon x\in D\cap \mathbb{Q}^d\right\},\]
in order to bound the stochastic error. Indeed, this technique leads to the following result.
\begin{theorem}\label{thm: general rate}
	Assume \hyperref[framework]{($\cA$)} and fix $\theta>0$. Then for any $j\in[d],$ the following holds true:
\begin{enumerate}[label={\alph*)},ref={\thetheorem\ \alph*)}]
\item \label{thm: general rate exp}
If \ref{ass: exp mom} is fulfilled, exists a constant $c>0$ depending on $\theta$, such that for large enough $T$ it holds for any $\h\in\mathsf{H}_T$
\[\forall p\in\Big[1,\theta\log\Big(\sum_{i=1}^d h_i^{-1}\Big)\Big]\colon\quad\mathcal{R}^{(p)}_\infty(b^j\rho,\bar{b}^{(j)}_{\h,T};D)\leq	\mathcal{B}_{b^j\rho}(\h)+c(TV(\h))^{-1/2}\log\left(\sum_{i=1}^d h_i^{-1}\right)^{1/2}. \]
\item \label{thm: rate trunc est}
If \ref{ass: th mom} is fulfilled and $\delta=\alpha T^{1/4}$ for fixed $\alpha>0,$ exists a constant $c>0$ depending on $\theta$ and $\alpha$, such that for large enough $T$ it holds for any $\h\in\mathsf{H}_T$
\[\forall p\in\Big[1,\theta\log\Big(\sum_{i=1}^d h_i^{-1}\Big)\Big]\colon\quad\mathcal{R}^{(p)}_\infty(b^j\rho,\bar{b}^{(j),\delta}_{\h,T};D)\leq	\mathcal{B}_{b^j\rho}(\h)+c(TV(\h))^{-1/2}\log\left(\sum_{i=1}^d h_i^{-1}\right)^{1/2}. \]
\end{enumerate}
\end{theorem}
The estimators defined in \eqref{def: aux est} can be interpreted as estimators for $b\rho,$ thus natural modifications for directly estimating the drift $b$ are given by the following Nadaraya--Watson type estimators defined for $x\in D, j\in[d],\delta>0, \h^b,\h^\rho\in(0,1)^d,$
\[\hat{b}^{(j)}_{\h^b,\h^\rho,T}(x)\coloneq \frac{	\bar{b}^{(j)}_{\h^b,T}(x)}{\vert \hat{\rho}_{\h^\rho,T}(x)\vert +r(T)},\quad \hat{b}^{(j),\delta}_{\h^b,\h^\rho,T}(x)\coloneq \frac{\bar{b}^{(j),\delta}_{\h^b,T}(x)}{\vert \hat{\rho}_{\h^\rho,T}(x)\vert +r(T)}, \]
where 
\begin{equation}
\label{def: dens est}
\hat{\rho}_{\h^\rho,T}(x)\coloneq \frac 1 T \int_0^T K_{\h^\rho}(x-X_s)\d s, 
\end{equation}
is a kernel density estimator, whose behaviour has been studied in \cite{dexheimeraihp} under \hyperref[framework]{($\cA$)}, and $r(T)$ is a strictly positive function. Applying Theorem \ref{thm: general rate} with the rate-optimal choices of $\h^b,\h^\rho$ together with suitable $r(T)$ then yields that both estimators achieve the nonparametric rate of convergence for a weighted version of the $\sup$-norm risk (see Remark \ref{rem: weight}).
\begin{corollary}\label{cor: rate}
Let $j\in[d],$ assume \hyperref[framework]{($\cA$)}, $b^j\rho,\rho\in\cH_D(\bbeta,\cL),$ where $\bbeta,\cL$ are given as in Definition \ref{def: hölder} with $\bar{\bbeta}>d/2\lor (2\land d)$ and $ \max_{i\in[d]}(\llfloor\beta_i\rrfloor)\leq \mathsf{k}$. Then choosing 
\begin{align*}
r(T)&= \Phi_{d,\bbeta}(T)\exp(\sqrt{\log(T)}),\quad\textrm{where}\quad \Phi_{d,\bbeta}\coloneq \begin{cases}
	\frac{\log(T)}{\sqrt{T}},& d\leq 2,\\
	\left(\frac{\log(T)}{T}\right)^{\frac{\bar{\bbeta}}{2\bar{\bbeta}+d-2}}, & d\geq 3,
\end{cases}
\\
\h^b&=(h_1^b(T),\ldots,h^b_d(T)),\quad\textrm{where}\quad h_i^b(T)\sim\left(\frac{\log(T)}{T}\right)^{\frac{\bar{\bbeta}}{(2\bar{\bbeta}+d)\beta_i}},
\\
\h^\rho&=(h^\rho_1(T),\ldots,h^\rho_d(T)),\quad\textrm{where}\quad h^\rho_i(T)\sim\begin{cases}
\frac{\log(T)^2}{\sqrt{T}},&d=1,
\\
\left(\frac{\log(T)^4}{T}\right)^{\frac{\bar{\bbeta}}{4\beta_i}}, &d=2,
\\
\left(\frac{\log(T)}{T}\right)^{\frac{\bar{\bbeta}}{(2\bar{\bbeta}+d-2)\beta_i}}, &d\geq3,
\end{cases}
\end{align*}
 gives the following:
\begin{enumerate}
\item[a)] \label{cor: a}If \ref{ass: exp mom} is fulfilled, it holds 
\[\E\left[\Vert(\hat{b}^{(j)}_{\h^b,\h^\rho,T}-b^j)\rho\Vert_{L^\infty(D)} \right]\in\cO\left(\left(\frac{\log(T)}{T}\right)^{\frac{\bar{\bbeta}}{2\bar{\bbeta}+d}}\right).\]
\item[b)]\label{cor: b} If \ref{ass: th mom} is fulfilled and $\delta=\alpha T^{1/4}$ for some fixed $\alpha>0,$ it holds 
\[\E\left[\Vert(\hat{b}^{(j),\delta}_{\h^b,\h^\rho,T}-b^j)\rho\Vert_{L^\infty(D)} \right]\in\cO\left(\left(\frac{\log(T)}{T}\right)^{\frac{\bar{\bbeta}}{2\bar{\bbeta}+d}}\right).\]
\end{enumerate}
\end{corollary}
\begin{remark}\label{rem: weight}
\begin{enumerate}[a)]
\item The requirement $\bar{\bbeta}>d/2$ is frequent in the context of drift estimation for diffusion type processes, see, for example, Theorem 13 in \cite{str15} or Theorem 4.5 in \cite{kindiff} (note that the investigated process in this reference is $2d$ dimensional). It stems from the term involving $T^{-1/4}$ in Propositions \ref{prop: cont moment bound}, \ref{prop: jump moment bound} and \ref{prop: trunc chain} dominating the uniform moment bounds if the bandwidths are too small, which is the case if $\bar{\bbeta}\leq d/2$. The other condition $\bar{\bbeta}>2\land d$ is needed for the analysis of the invariant density estimator $\hat{\rho}$ and has been discussed in detail in Remark 4.3 of \cite{dexheimeraihp}. 
\item \textcolor{dex}{The weighted version of the $\sup$-norm risk considered in Corollary \ref{cor: rate} can be interpreted as the empirical counterpart to the traditional $\sup$-norm risk, similarly as the empirical $L^2$ risk investigated in \cite{schmisser14} compared to the standard $L^2$ risk. Clearly, if $\rho$ is bounded away from zero on the domain $D,$ the results of Corollary \ref{cor: rate} also translate to the normal $\sup$-norm risk. Intuitively, the weight function $\rho$ is needed since $\rho$ being small on a subset of $D$ corresponds to the process $\X$  spending little time in that subset, and thus providing limited information about its drift there.}
\item\textcolor{dex}{Although the estimator $\hat{b}^{(j)}_{\h^b,\h^\rho,T}$ requires more restrictive assumptions than $\hat{b}^{(j),\delta}_{\h^b,\h^\rho,T}$ for achieving the same rate of convergence, its investigation can also be justified apart from the theoretical comparison of the continuous and discontinuous setting. This is because if one wants to modify the given estimators for practical applications, in particular under discrete observations, the needed knowledge on $\X$'s jumps for the truncated estimator becomes challenging, whereas employing a discretised version of $\hat{b}^{(j)}_{\h^b,\h^\rho,T}$ is relatively straightforward. Nevertheless, since $\hat{b}^{(j),\delta}_{\h^b,\h^\rho,T}$ only requires insight on the jumps of magnitude larger than $\sim T^{1/4},$ their detection is possible for large enough $T$ through a similar truncation method as used in \cite{maiou}.  }
\end{enumerate}
\end{remark}
\section{Adaptive Estimation}\label{sec: adap}
As the rate-optimal choices $\h^b$ and $\h^\rho$ in Corollary \ref{cor: rate} obviously depend on the in general unknown smoothness indices, the question of finding an adaptive estimator suited for drift estimation without this knowledge naturally arises. For this we employ a variant of the Goldenshluger--Lepski procedure presented in \cite{goldenshluger2011,lepski2013}, which requires us to introduce the following setup.
As in Section \ref{sec: drift est} we fix a bounded open set $D\subset\R^d,$ on which we additionally assume that $\rho\geq \rho_{\min}$ holds for some a priori known constant $\rho_{\min}>0$. For $K$ given as in Section \ref{sec: drift est} and any multi-bandwidths $\h=(h_1,\ldots,h_d),\boldsymbol{\eta}=(\eta_1,\ldots,\eta_d)\in(0,1)^d$ and $x=(x_1,\ldots,x_d)^\top\in\R^d$ we set
\begin{align}\notag
K_{\h}\star K_{\boldsymbol{\eta}}(x)&\coloneq \prod_{i=1}^d (h_i^{-1} k(x_i/h_i))\ast (\eta_i^{-1} k(x_i/\eta_i))
\\&= \prod_{i=1}^d  \int_{\R}  h_i^{-1} k((u-x_i)/h_i)\eta_i^{-1} k(u/\eta_i)\d u\notag
\\
\notag &=\int_{\R^d} K_{\h}(u-x)K_{\boldsymbol{\eta}}(u)\d u,
\end{align}
and for $(X^{j,\delta}_t)_{t\geq0},$ defined as in Section \ref{sec: drift est} we let
\begin{align*}
\bar{b}^{(j)}_{\h,\boldsymbol{\eta},T}(x)&\coloneq \frac{1}{T}\int_0^T (K_{\h}\star K_{\boldsymbol{\eta}})(X_{s-}-x)\d X^j_s,
\\
\bar{b}^{(j),\delta}_{\h,\boldsymbol{\eta},T}(x)&\coloneq \frac{1}{T}\int_0^T (K_{\h}\star K_{\boldsymbol{\eta}})(X_{s-}-x)\d X^{j,\delta}_s.
\end{align*}
Additionally we introduce for some fixed, but arbitrary, $\iota>1,$ the set of candidate multi-bandwidths as
\[\mathscr{H}_T\coloneq \left\{\h=(h_1,\ldots,h_d)\in(0,\log(1+T)^{-1})^d\colon h_i=\iota^{-k_i}\,\mathrm{ with }\, k_i\in\N, \iota^{\sum_{i=1}^dk_i}\log\left(\sum_{i=1}^d \iota^{k_i}\right)^{4}\leq T^{1/2}  \right\}, \]
and for any $\h\in\mathscr{H}_T,\theta,\alpha>0,$ we denote
\begin{align*}
A_{T,1}(\h,\theta)&\coloneq2\e T^{-1/2} V(\h)^{-1/2}\log\left(\sum_{i=1}^d h_i^{-1}\right)^{1/2}\Vert K\Vert_\infty\Bigg(C_1\left(21+29d\sqrt{\Vert a\Vert_{\infty}}
+17c_1 \right) 
\\&\quad+\sqrt{\theta}\tilde{c}_2\left(3+4\sqrt{\Vert a_{jj}\Vert_\infty }
+3c_1\right)\Bigg)
\\
A_{T,2}(\h,\theta,\alpha)&\coloneq
2\e T^{-1/2} V(\h)^{-1/2}\log\left(\sum_{i=1}^d h_i^{-1}\right)^{1/2}\Vert K\Vert_\infty\Bigg(29C_1d\sqrt{\Vert a\Vert_{\infty}}
+4\sqrt{\theta\Vert a_{jj}\Vert_\infty }\tilde{c}_2,
\\
&\quad +(1+\cE_\alpha)\left(64C_1 \sqrt{d}+3\sqrt{\theta}\tilde{c}_2\right)
+\frac{\Vert \gamma\Vert_\infty}{\gamma_{\min}}\left(1+\frac{\nu_2\Vert \gamma\Vert_\infty^2}{2}\exp(\alpha)\right)\left(64\sqrt{d}C_1+3\sqrt{\theta}\tilde{c}_2\right)\Bigg),
\\
\Upsilon_{T,1}^j(\h)&\coloneq \sup_{\eta\in \mathscr{H}_T}\left(\Vert \bar{b}^{(j)}_{\h,\boldsymbol{\eta},T}-\bar{b}^{(j)}_{\boldsymbol{\eta},T}\Vert_{D,\infty}-A_{T,1}(\boldsymbol{\eta},d) \right)_+,
\\
\Upsilon_{T,2}^j(\h)&\coloneq \sup_{\eta\in \mathscr{H}_T}\left(\Vert \bar{b}^{(j),\delta}_{\h,\boldsymbol{\eta},T}-\bar{b}^{(j),\delta}_{\boldsymbol{\eta},T}\Vert_{D,\infty}-A_{T,2}(\boldsymbol{\eta},d,\alpha) \right)_+,
\end{align*}
where $C_1$ is introduced in the proof of Proposition \ref{prop: jump moment bound}, $c_1$ and $\tilde{c}_2$ are given in Lemma \ref{lemma: qv jump bound}, respectively Proposition \ref{prop: cont moment bound}, $\gamma_{\min}$ is defined in Assumption \ref{ass: th mom} and $\nu_2$ and $\cE_\alpha$ are denoted in Proposition \ref{prop: trunc chain}. Choosing the bandwidths $\hat{\h}^j_1,\hat{\h}^j_2$ according to
\begin{align*}
	\Upsilon_{T,1}^j(\hat{\h}^j_1)+A_{T,1}(\hat{\h}^j_1,d)&=\inf_{\h\in\mathscr{H}_T}\left(\Upsilon_{T,1}^j(\h)+A_{T,1}(\h,d) \right),
	\\
	\Upsilon_{T,2}^j(\hat{\h}^j_2)+A_{T,2}(\hat{\h}^j_2,d,\alpha)&=\inf_{\h\in\mathscr{H}_T}\left(\Upsilon_{T,2}^j(\h)+A_{T,2}(\h,d,\alpha) \right),
\end{align*}
then allows us to provide the following oracle inequalities over the class of candidate bandwidths $\mathscr{H}_T$ for the adaptive versions of the auxiliary estimators introduced in \eqref{def: aux est}.
\begin{proposition}\label{thm: adap}Assume \hyperref[framework]{($\cA$)}. Then for any $j\in[d]$ the following holds true:
\begin{enumerate}
\item[a)] If \ref{ass: exp mom} is fulfilled then there exists a constant $c>0,$ such that for large enough $T$
\[\mathcal{R}^{(1)}_\infty(b^j\rho,\bar{b}^{(j)}_{\hat{\h}^j_1,T};D)\leq c\left(\inf_{\h\in\mathscr{H}_T}\left(\mathcal{B}_{b^j\rho}(\h)+(TV(\h))^{-1/2}\log\left(\sum_{i=1}^dh_i^{-1}\right)^{1/2}\right)+\log(T)^{d+1/2}T^{-1/2}\right). \]
\item[b)] If \ref{ass: th mom} is fulfilled and $\delta=\alpha T^{1/4}$ then there exists a constant $c>0,$ such that for large enough $T$
\[\mathcal{R}^{(1)}_\infty(b^j\rho,\bar{b}^{(j),\delta}_{\hat{\h}^j_2,T};D)\leq c\left(\inf_{\h\in\mathscr{H}_T}\left(\mathcal{B}_{b^j\rho}(\h)+(TV(\h))^{-1/2}\log\left(\sum_{i=1}^dh_i^{-1}\right)^{1/2}\right)+\log(T)^{d+1/2}T^{-1/2}\right). \]
\end{enumerate}
\end{proposition}
Similarly as in Section \ref{sec: drift est}, we now introduce Nadaraya--Watson type estimators, more precisely for $x\in D, j\in[d],\delta>0$, we  set
\begin{align}
	\hat{b}_{\textrm{adap},T}^{(j)}(x)&\coloneq \frac{\bar{b}^{(j)}_{\hat{\h}^j_1,T}(x)}{ \hat{\rho}_{\hat{\h}_\rho,T}(x)\lor \rho_{\min}}, \quad
	\hat{b}_{\textrm{adap},T}^{(j),\delta}(x)\coloneq \frac{\bar{b}^{(j),\delta}_{\hat{\h}^j_2,T}(x)}{ \hat{\rho}_{\hat{\h}_\rho,T}(x)\lor \rho_{\min}}.\label{def: adap est}
\end{align}
Here for a multi-bandwidth $\h$ the kernel density estimator $\hat{\rho}_{\h,T}$ is defined as in \eqref{def: dens est}, and $\hat{\h}_\rho$ is assumed to be an adaptive bandwidth choice, such that 
\begin{equation}\label{eq: adap dens}
\cR^{(1)}_\infty(\rho,\hat{\rho}_{\hat{\h}_\rho,T};D)\in \cO\left(\left(\frac{\log(T)}{T}\right)^{\frac{\bar{\bbeta}}{2\bar{\bbeta}+d}}\right), 
\end{equation}
if $\rho\in\cH_D(\bbeta,\cL)$ with $\bar{\bbeta}>d\land 2$ and \hyperref[framework]{($\cA$)} is fulfilled. This assumption is discussed in detail in Remark \ref{rem: adap}.
As the optimal bandwidth choices in Corollary \ref{cor: rate} are contained in $\mathscr{H}_T$ if $\bar{\bbeta}>d/2,$ a straightforward application of the above oracle inequalities then yields the following result on the $\sup$-norm risk of the adaptive Nadaraya--Watson type estimators defined in \eqref{def: adap est}.
\begin{theorem}\label{thm: adap final} Let $j\in[d],$ assume \hyperref[framework]{($\cA$)}, $b^j\rho,\rho\in\cH_D(\bbeta,\cL),$ with $\bar{\bbeta}>d/2\lor (2\land d)$ and $ \max_{i\in[d]}(\llfloor\beta_i\rrfloor)\leq \mathsf{k}$. Then the following holds true:
\begin{enumerate}[a)]
\item If \ref{ass: exp mom} is fulfilled it holds
\begin{equation}\notag 
\mathcal{R}^{(1)}_\infty(b^j,\hat{b}^{(j)}_{\mathrm{adap},T};D)\in\cO\left( \left(\frac{\log(T)}{T}\right)^{\frac{\bar{\bbeta}}{2\bar{\bbeta}+d}}\right).
\end{equation}
\item If \ref{ass: th mom} is fulfilled and $\delta=\alpha T^{1/4}$ it holds
\begin{equation}\notag 
	\mathcal{R}^{(1)}_\infty(b^j,\hat{b}^{(j),\delta}_{\mathrm{adap},T};D)\in\cO\left( \left(\frac{\log(T)}{T}\right)^{\frac{\bar{\bbeta}}{2\bar{\bbeta}+d}}\right).
\end{equation}
\end{enumerate}
\end{theorem}
\begin{remark}\label{rem: adap}
In the isotropic setting, where $\rho$ is assumed to be in $\cH_D(\bbeta,\cL)$ with $\beta_i=\beta_1$ for all $i\in[d],$ the assumption on the existence of an adaptive bandwidth choice $\hat{\h}_\rho$ satisfying \eqref{eq: adap dens} can be justified through Theorem 4.2 in \cite{dexheimeraihp} as soon as \hyperref[framework]{($\cA$)} and $\beta_1>2$ are fulfilled and $d\geq 3$. The assumption $d\geq3$ in this context is no restriction as the optimal bandwidth choice $\h^\rho$ given in Corollary \ref{cor: rate} is independent of $\bbeta$ in the isotropic setting if $d<3$. In fact, this reference provides $\hat{\h}_\rho,$ such that
\[ \cR^{(1)}_\infty(\rho,\hat{\rho}_{\hat{\h}_\rho,T};D)\in\cO\left(\log(T)\Phi_{d,\bbeta} \right),\]
where $\Phi_{d,\bbeta}$ is defined in Corollary \ref{cor: rate}, thus achieving better results than needed for our analysis. Using the uniform moment bounds of Theorem 3.1 in \cite{dexheimeraihp} and setting up a variant of the Goldenshluger--Lepski procedure analogous to the given section, it is straightforward to extend these results to the anisotropic case. Hence, the assumption on the existence of a suitable adaptive bandwidth choice for the kernel density estimator can also be justified in the general anisotropic setting.
\end{remark}
\textcolor{dex}{
\section{Conclusion}\label{sec: dis}
We have introduced two different adaptive Nadaraya--Watson type estimators for the drift of a jump diffusion of the form \eqref{eq: jumpdiff intro} and shown that they both achieve the classical nonparametric rate of convergence measured in the $\sup$-norm risk under anisotropic Hölder smoothness conditions and differing assumptions on the jump measure. As explained in Section \ref{sec: intro}, we investigate both estimators in detail for comparing the continuous with the discontinuous setting. We showed that under the exponential moment assumption \ref{ass: exp mom}, the same estimator as used in the continuous setting in \cite{aeck22} still achieves satisfying results and were able to generalise this to the less restrictive third moment assumption \ref{ass: th mom} by introducing a truncated estimator. Crucial for the proofs were new uniform moment bounds for empirical processes related to jump diffusions, which are proven through the extension of Talagrand's generic chaining device in \cite{dirksen15}, together with the exponential martingale inequality \eqref{eq: exp martingale ineq}. Since our results on the rate of convergence coincide with the benchmark case of continuous and reversible diffusion processes, in which the classical nonparametric rate of convergence is known to be minimax optimal (see \cite{str15}), our results can be regarded as optimal as well.  }
\appendix
\section{Auxilliary results}
Before we can start with the proofs of the main results of this paper we require some auxiliary results. Firstly, we introduce the following Bernstein type inequality for exponentially $\beta$-mixing Markov processes.
\begin{lemma}\label{lemma: bernstein dirksen form}
	Suppose that $\X$ is an exponentially $\beta$-mixing Markov process, and let $g$ be a bounded, measurable function satisfying $\mu(g)=0$. Then, for any $T,u>0$ and $m_T\in(0,T/4]$ exists $\tau \in [m_T,2m_T],$ such that 
	\begin{align*}
		&\P\left(\frac{1}{\sqrt{T}}\int_0^T g(X_s)\d s>32\sqrt{u}\left(\sqrt{\Var\left(\frac{1}{\sqrt{\tau}}\int_0^\tau g(X_s)\d s \right)}+2\sqrt{u}\Vert g\Vert_\infty \frac{m_T}{\sqrt{T}} \right) \right)
		\\
		&\leq 2\exp(-u)+\frac{T}{m_T}c_\kappa \exp(-\kappa m_T)\1_{(u,\infty)}\left(\frac{T}{16m_T}\right)
	\end{align*}
\end{lemma}
\begin{proof}
By Lemma 3.1 in \cite{dexheimer2020mixing} we have that for any $T,u>0$ and $m_T\in(0,T/4]$ there exists $\tau \in [m_T,2m_T],$ such that 
\begin{align*}
\P\left(\frac{1}{\sqrt{T}}\int_0^T g(X_s)\d s>u \right)&\leq 2\exp\left(-\frac{u^2}{32\left(\Var\left(\frac{1}{\sqrt{\tau}}\int_0^\tau g(X_s)\d s \right)+2u\Vert g\Vert_\infty \frac{m_T}{\sqrt{T}}\right)} \right)
\\
&\quad +\frac{T}{m_T}c_\kappa \exp(-\kappa m_T)\1_{(0,4\sqrt{T}\Vert g\Vert_\infty)}(u).
\end{align*}
Hence we get
\begin{align*}
	&\P\left(\frac{1}{\sqrt{T}}\int_0^T g(X_s)\d s>32\sqrt{u}\left(\sqrt{\Var\left(\frac{1}{\sqrt{\tau}}\int_0^\tau g(X_s)\d s \right)}+2\sqrt{u}\Vert g\Vert_\infty \frac{m_T}{\sqrt{T}} \right) \right)
	\\
	&\leq 2\exp\left(-32u\frac{\Var\left(\frac{1}{\sqrt{\tau}}\int_0^\tau g(X_s)\d s \right)+4u\Vert g\Vert^2_\infty \frac{m^2_T}{T}+4\sqrt{u}\Var\left(\frac{1}{\sqrt{\tau}}\int_0^\tau g(X_s)\d s \right)\Vert g\Vert_\infty \frac{m_T}{\sqrt{T}} }{\Var\left(\frac{1}{\sqrt{\tau}}\int_0^\tau g(X_s)\d s\right)+64\sqrt{u}\Var\left(\frac{1}{\sqrt{\tau}}\int_0^\tau g(X_s)\d s\right)\Vert g\Vert_\infty \frac{m_T}{\sqrt{T}} +128u\Vert g\Vert^2_\infty \frac{m_T^2}{T}} \right)
	\\
	&\quad +\frac{T}{m_T}c_\kappa \exp(-\kappa m_T)\1_{(0,4\sqrt{T}\Vert g\Vert_\infty)}\left(32\sqrt{u}\left(\sqrt{\Var\left(\frac{1}{\sqrt{\tau}}\int_0^\tau g(X_s)\d s \right)}+2\sqrt{u}\Vert g\Vert_\infty \frac{m_T}{\sqrt{T}} \right)  \right)
	\\
	&\leq 2\exp(-u)+\frac{T}{m_T}c_\kappa \exp(-\kappa m_T)\1_{(0,4\sqrt{T}\Vert g\Vert_\infty)}\left(64u\Vert g\Vert_\infty \frac{m_T}{\sqrt{T}}  \right)
		\\
	&\leq 2\exp(-u)+\frac{T}{m_T}c_\kappa \exp(-\kappa m_T)\1_{(u,\infty)}\left(\frac{T}{16m_T}\right).
\end{align*}
\end{proof}
The two following lemmas will prove to be useful in combination with the exponential martingale inequality \ref{eq: exp martingale ineq}.
\begin{lemma}\label{lemma: qv jump bound}
Grant assumption \ref{ass: exp mom}. Then for any $0<y\leq c_{1,\nu}(2\Vert g\Vert_\infty \Vert \gamma\Vert_\infty)^{-1},j\in[d]$ it holds
\[\int_0^T\int_{\R^d}\exp(yg(X_{s-})\gamma^j(X_{s-})z))-1-yg(X_{s-})\gamma^j(X_{s-})z)\nu(\d z)\d s\leq c_1 y^2\int_0^Tg(X_s)^2\d s, \]
where
\begin{equation}\label{eq:  def c_1}
c_1\coloneq\Vert \gamma\Vert_\infty^2\left( \frac{1}{2}\exp(c_{1,\nu}/2) \int_{B(0,1)} \Vert z\Vert^2\nu(\d z)+4 c_{1,\nu}^{-2}\int_{B(0,1)^\mathsf{C}}\exp(c_{1,\nu}\Vert z\Vert)\nu(\d z)\right).
\end{equation}
\end{lemma}
\begin{proof}
Define $h\colon \R\to \R, h(x)=\exp(x)-1-x$. Then $h$ is monotonically increasing on $[0,\infty)$ and the classical inequality $\exp(x)\geq 1+x, x\in\R$ gives for any $x\in\R$
\begin{align*}
0=\int_{-\vert x\vert}^{\vert x\vert} y\d y\leq \int_{-\vert x\vert}^{\vert x\vert} \exp(y)-1\d y=\exp(\vert x\vert)-\exp(-\vert x\vert)-2\vert x\vert.
\end{align*}
Rearranging this inequality yields $h(x)\leq h(\vert x\vert)$ for any $x\in\R$.
By Taylor's theorem we then get for $c_\gamma\coloneq\Vert \gamma\Vert_\infty,$
\begin{align*}
	\int_{\R^d}h(yg(X_{s-})\gamma^j(X_{s-})z))\nu(\d z)
	&\leq  \int_{B(0,1)}h(yc_\gamma\vert g(X_{s-})\vert \Vert z\Vert)\nu(\d z)+\int_{B(0,1)^\mathsf{C}}h(yc_\gamma\vert g(X_{s-})\vert \Vert z\Vert)\nu(\d z)
	\\
	&\leq \frac{1}{2}\exp(yc_\gamma\vert g(X_{s-})\vert)y^2c^2_\gamma\vert g(X_{s-})\vert^2 \int_{B(0,1)} \Vert z\Vert^2\nu(\d z)\\&\quad+\int_{B(0,1)^\mathsf{C}}\exp(yc_\gamma\vert g(X_{s-})\vert \Vert z\Vert)-1-yc_\gamma\vert g(X_{s-})\vert \Vert
	z\Vert\nu(\d z).
\end{align*}
Now by \ref{ass: exp mom} we get by Taylor's theorem and the assumption on $y$,
\begin{align*}
	&\int_{B(0,1)^\mathsf{C}}\exp(yc_\gamma\vert g(X_{s-})\vert \Vert z\Vert)-1-yc_\gamma\vert g(X_{s-})\vert \Vert
	z\Vert\nu(\d z)
\\
&\leq \frac{1}{2}(yc_\gamma\vert g(X_{s-})\vert)^2\int_{B(0,1)^\mathsf{C}}\exp(yc_\gamma\vert g(X_{s-})\vert \Vert z\Vert) \Vert
z\Vert^2\nu(\d z)
\\
&\leq  4(yc_\gamma\vert g(X_{s-})\vert c_{1,\nu}^{-1})^2\int_{B(0,1)^\mathsf{C}}\exp(c_{1,\nu}\Vert z\Vert)\nu(\d z),
\end{align*}
where we used the inequality $\exp(x)\geq x^2/2,x>0$ in the last step.
Hence if $2y c_\gamma\vert g(X_{s-})\vert\leq c_{1,\nu},$ it holds
\begin{align*}
	&\int_{\R^d}\exp(yg(X_{s-})\gamma^j(X_{s-})z))-1-yg(X_{s-})\gamma^j(X_{s-})z)\nu(\d z)
	\\
	&\leq y^2 g(X_{s-})^2c_\gamma^2\left( \frac{1}{2}\exp(c_{1,\nu}/2) \int_{B(0,1)} \Vert z\Vert^2\nu(\d z)+4 c_{1,\nu}^{-2}\int_{B(0,1)^\mathsf{C}}\exp(c_{1,\nu}\Vert z\Vert)\nu(\d z)\right),
\end{align*}
which concludes the proof.
\end{proof}
\begin{lemma}\label{lemma: qv jump bound comp supp}
Grant assumption \ref{ass: th mom}. Then, for any $y,\delta>0$ it holds
\begin{align*}
		&\int_{B(0,\delta)}\exp(yg(X_{s-})\gamma^j(X_{s-})z))-1-yg(X_{s-})\gamma^j(X_{s-})z)\nu(\d z)
	\\
	&\leq (y\Vert \gamma\Vert_\infty g(X_{s-}))^2\frac{\nu_2\exp(y\Vert \gamma\Vert_\infty\vert g(X_{s-})\vert \delta)}{2}.
\end{align*}	
\end{lemma}
\begin{proof}
Arguing as in the proof of Lemma \ref{lemma: qv jump bound} Taylor's theorem gives for $c_\gamma=\Vert \gamma\Vert_\infty,$
\begin{align*}
	&\int_{B(0,\delta)}\exp(yg(X_{s-})\gamma^j(X_{s-})z))-1-yg(X_{s-})\gamma^j(X_{s-})z)\nu(\d z)
	\\
	&\leq \int_{B(0,\delta)}\exp(yc_\gamma\vert g(X_{s-})\vert \Vert z\Vert)-1-yc_\gamma\vert g(X_{s-})\vert \Vert z\Vert\nu(\d z)
		\\
	&\leq (yc_\gamma\vert g(X_{s-})\vert)^2\frac{\exp(yc_\gamma\vert g(X_{s-})\vert \delta)}{2} \int_{B(0,\delta)} \Vert z\Vert^2\nu(\d z).
\end{align*}
\end{proof}
\section{Proofs for Section \ref{sec: unif mom}}
We start with the proof of Proposition \ref{prop: jump moment bound} as it is more involved than the proof of Proposition \ref{prop: cont moment bound}.
\begin{proof}[Proof of Proposition \ref{prop: jump moment bound}]
	Firstly, Theorem 2.2 in \cite{applebaum09paper} gives for any $T,x,y>0$
	\begin{align*}
		\P\left(\mathbb{J}^j_T(g)>\frac{x}{\sqrt{T}}+\frac{1}{\sqrt{T}y} \int_0^T\int_{\R^d}\exp(yg(X_{s-})\gamma^j(X_{s-})z))-1-yg(X_{s-})\gamma^j(X_{s-})z)\nu(\d z)\d s \right)\leq \exp(-xy).
	\end{align*}
	Thus, applying Lemma \ref{lemma: qv jump bound} to $g$ and $-g$ we get for any $x,T>0$ and $0<y\Vert g\Vert_\infty \Vert\gamma^j\Vert_\infty\leq c_{1,\nu}/2$ 
	\begin{align*}
		\P\left(\left\vert\mathbb{J}^j_T(g)\right\vert >\frac{x}{\sqrt{T}}+\frac{yc_1}{\sqrt{T}}\int_0^Tg(X_s)^2\d s \right)
		&\leq 2\exp(-xy),
	\end{align*}
	which implies for any $r>0,$
	\begin{align*}
		&\P\left(\left\vert\mathbb{J}^j_T(g)\right\vert >\frac{x}{\sqrt{T}}+yc_1r+yc_1\sqrt{T}\mu(g^2)\right)
		\\
		&\leq 
		\P\left(\left\vert\mathbb{J}^j_T(g)\right\vert >\frac{x}{\sqrt{T}}+yc_1r+yc_1\sqrt{T}\mu(g^2),\frac{1}{\sqrt{T}}\int_0^Tg(X_s)^2-\mu(g^2)\d s\leq r \right)
		\\&\quad +\P\left(\frac{1}{\sqrt{T}}\int_0^Tg(X_s)^2-\mu(g^2)\d s>r \right)
		\\
		&\leq 
		\P\left(\left\vert\mathbb{J}^j_T(g)\right\vert >\frac{x}{\sqrt{T}}+\frac{yc_1}{\sqrt{T}}\int_0^Tg(X_s)^2\d s \right)+\P\left(\frac{1}{\sqrt{T}}\int_0^Tg(X_s)^2-\mu(g^2)\d s>r \right)
		\\
		&\leq 
		2\exp(-xy)+\P\left(\frac{1}{\sqrt{T}}\int_0^Tg(X_s)^2-\mu(g^2)\d s>r \right).
	\end{align*}
	Hence choosing $m_T=\sqrt{T}/(2\sqrt{\kappa})$ in Lemma \ref{lemma: bernstein dirksen form} we have for large enough $T,$
	\begin{align}
		&\P\left(\left\vert\mathbb{J}^j_T(g)\right\vert >\frac{x}{\sqrt{T}}+yc_132\sqrt{u}\left(\sqrt{\Var\left(\frac{1}{\sqrt{\tau}}\int_0^\tau g(X_s)^2\d s \right)}+\kappa^{-1/2}\sqrt{u}\Vert g^2-\mu(g^2)\Vert_\infty  \right) +yc_1\sqrt{T}\mu(g^2)\right)\notag
		\\
		&\leq \notag
		2\exp(-xy)+2\exp(-u)+2c_\kappa\sqrt{\kappa T} \exp(-\sqrt{\kappa T}/2)\1_{(u,\infty)}\left(\frac{\sqrt{\kappa T}}{8}\right)
		\\
		&\leq 
		2\exp(-xy)+4\exp(-u).\label{eq: jump chain 1}
	\end{align}
	Now note, that for $u>0$ it holds for $x_u=u \left(\sqrt{T\mu(g^2)/u}+T^{3/8}\Vert g\Vert_{L^4(\mu)}+T^{1/4}\Vert g\Vert_\infty  \right), y_u=u/x_u$
	\begin{align*}
		&\frac{x_u}{\sqrt{T}}+y_uc_132\sqrt{u}\left(\sqrt{\Var\left(\frac{1}{\sqrt{\tau}}\int_0^\tau g(X_s)^2\d s \right)}+\kappa^{-1/2}\sqrt{u}\Vert g^2-\mu(g^2)\Vert_\infty  \right) +y_uc_1\sqrt{T}\mu(g^2)
		\\
		&\leq \sqrt{u}(1+c_1)\Vert g\Vert_{L^2(\mu)}+uT^{-1/8}\Vert g\Vert_{L^4(\mu)}+uT^{-1/4}\Vert g\Vert_\infty  \\&\quad+32y_uc_1\sqrt{u}\left(\sqrt{\tau\mu(g^4)}+\kappa^{-1/2}\sqrt{u}(\Vert g^2\Vert_\infty+\mu(g^2))  \right) 
		\\
		&\leq \sqrt{u}(1+c_1)\Vert g\Vert_{L^2(\mu)}+uT^{-1/8}\Vert g\Vert_{L^4(\mu)}+uT^{-1/4}\Vert g\Vert_\infty  \\&\quad+32y_uc_1\sqrt{u}\left((T/\kappa)^{1/4}\sqrt{\mu(g^4)}+2\kappa^{-1/2}\sqrt{u}\Vert g^2\Vert_\infty  \right) 
		\\
		&\leq \sqrt{u}(1+c_1)\Vert g\Vert_{L^2(\mu)}+uT^{-1/8}\Vert g\Vert_{L^4(\mu)}+uT^{-1/4}\Vert g\Vert_\infty  \\&\quad+32c_1\sqrt{u}\kappa^{-1/4}T^{-1/8}\Vert g\Vert_{L^4(\mu)}+64uc_1\kappa^{-1/2}T^{-1/4}\Vert g\Vert_\infty
		\\
		&=u\left(T^{-1/8}\Vert g\Vert_{L^4(\mu)}+(64c_1\kappa^{-1/2}+1)T^{-1/4}\Vert g\Vert_\infty  \right)+\sqrt{u}\left((1+c_1)\Vert g\Vert_{L^2(\mu)}+32c_1\kappa^{-1/4}T^{-1/8}\Vert g\Vert_{L^4(\mu)}\right)
	\end{align*}
	Since $2y_u\Vert g\Vert_\infty \Vert \gamma\Vert_\infty \leq c_{1,\nu}$ holds for large enough $T$ for any $u>0$, this implies for large enough values of $T$ together with \eqref{eq: jump chain 1}
	\begin{align*}
		&\P\Bigg(\left\vert\mathbb{J}^j_T(g)\right\vert > u\left(T^{-1/8}\Vert g\Vert_{L^4(\mu)}+(64c_1\kappa^{-1/2}+1)T^{-1/4}\Vert g\Vert_\infty  \right)\\&\hskip 2cm +\sqrt{u}\left((1+c_1)\Vert g\Vert_{L^2(\mu)}+32c_1\kappa^{-1/4}T^{-1/8}\Vert g\Vert_{L^4(\mu)}\right)\Bigg)
		\\
		&\leq 
		6\exp(-u).
	\end{align*}
	As 
	\begin{align}\label{eq: conc jump}
		\notag	&\P\Bigg(\left\vert\mathbb{J}^j_T(g)\right\vert > 3u\left(T^{-1/8}\Vert g\Vert_{L^4(\mu)}+(64c_1\kappa^{-1/2}+1)T^{-1/4}\Vert g\Vert_\infty  \right)\\&\hskip 2cm +\sqrt{3u}\left((1+c_1)\Vert g\Vert_{L^2(\mu)}+32c_1\kappa^{-1/4}T^{-1/8}\Vert g\Vert_{L^4(\mu)}\right)\Bigg)
		\\\notag
		&\leq 
		2\exp(-u),
	\end{align}
	is trivially true for $u\leq \log(2),$ we then obtain for large enough $T,$ that \eqref{eq: conc jump} holds true for any $u>0$. 
	Hence we can apply Theorem 3.5 of \cite{dirksen15}, which yields for any $p\geq 1$
	\begin{align*}
		&\E\Big[\sup_{g\in\mathcal{G}}\vert \mathbb{J}^j_T(g)\vert^p\Big]^{1/p}
		\\
		&\leq c\Bigg(\tilde{c}_1\int_0^\infty \log\left(\mathcal{N}(u,\mathcal{G},3T^{-1/8}d_{L^4(\mu)}+3(64c_1\kappa^{-1/2}+1)T^{-1/4}d_\infty) \right) \d u
		\\
		&\quad +\tilde{c}_2\int_0^\infty\sqrt{\log\left(\mathcal{N}(u,\mathcal{G},\sqrt{3}(1+c_1)d_{L^2(\mu)}+\sqrt{3072}c_1\kappa^{-1/4}T^{-1/8}d_{L^4(\mu)}) \right) }\d u\Bigg)
		+2\sup_{g\in\mathcal{G}}\E\left[\vert \mathbb{J}^j_T(g)\vert^p \right]^{1/p},
	\end{align*}
	where we bounded the $\gamma_\alpha$ functionals by the corresponding entropy integrals (see e.g. Equation (2.3) in \cite{dirksen15}). Noting that a slight adjustment of Lemma A.2 in \cite{dirksen15} implies for any $p\geq1,g\in\mathcal{G},$ together with \eqref{eq: conc jump} 
	\begin{align*}
		&\E\left[\vert \mathbb{J}^j_T(g)\vert^p \right]^{1/p}
		\\
		&\leq 6p\e^{1/(2\e)}(\sqrt{8\pi}\e^{1/(12p)})^{1/p}\e^{-1}\left(T^{-1/8}\Vert g\Vert_{L^4(\mu)}+(64c_1\kappa^{-1/2}+1)T^{-1/4}\Vert g\Vert_\infty  \right)
		\\&\quad+\sqrt{3p}2\e^{-1/2}(\sqrt{2\pi}\e^{1/(6p)})^{1/p}\e^{1/(2\e)}\left((1+c_1)\Vert g\Vert_{L^2(\mu)}+32c_1\kappa^{-1/4}T^{-1/8}\Vert g\Vert_{L^4(\mu)}\right),
	\end{align*}
	concludes the proof.

\end{proof}
\begin{proof}[Proof of Proposition \ref{prop: cont moment bound}]
	By Bernstein's inequality for continuous martingales (see e.g.\ p.153 in \cite{revuzyor1999}) and Lemma \ref{lemma: bernstein dirksen form} we have that there exists $\tau\in [\sqrt{T}/(2\sqrt{\kappa}),\sqrt{T}/\sqrt{\kappa}]$ such that for any $x>0$
	\begin{align*}
		&\P\left(\sqrt{T}\mathbb{M}^j_T(g)>\sqrt{T}x\right)
		\\
		&\leq \P\Bigg(\sqrt{T}\mathbb{M}^j_T(g)>\sqrt{T}x, \qv{\sqrt{T}\mathbb{M}^j_{\cdot}(g)}_T\leq 32\sqrt{Tu}\left(\sqrt{\Var\left(\frac{1}{\sqrt{\tau}}\int_0^\tau g^2(X_s)a_{jj}(X_s)\d s \right)}+2\sqrt{u}\Vert g^2a_{jj}\Vert_\infty \frac{m_T}{\sqrt{T}} \right) \\&\quad+\sqrt{T}\mu(g^2a_{jj})\Bigg) +2\exp(-u)+\frac{T}{m_T}c_\kappa \exp(-\kappa m_T)\1_{(u,\infty)}\left(\frac{T}{16m_T}\right)
		\\
		&\leq 2\exp\left(-\frac{Tx^2}{64\sqrt{Tu}\left(\sqrt{\Var\left(\frac{1}{\sqrt{\tau}}\int_0^\tau g^2(X_s)a_{jj}(X_s)\d s \right)}+\sqrt{u}\Vert g^2a_{jj}\Vert_\infty \frac{1}{\sqrt{\kappa}} \right)+2\sqrt{T}\mu(g^2a_{jj})}\right)+4\exp(-u)
		\\
		&\leq 2\exp\left(-\frac{\sqrt{T}x^2}{64\sqrt{u}\left(\left(\frac{T}{\kappa}\right)^{1/4}\sqrt{2\mu(g^4a^2_{jj})}+\sqrt{u}\Vert g^2a_{jj}\Vert_\infty \frac{1}{\sqrt{\kappa}} \right)+2\sqrt{T}\mu(g^2a_{jj})}\right)+4\exp(-u)
		\\
		&\leq 2\exp\left(-\frac{\sqrt{T}x^2}{64\Vert a_{jj}\Vert_\infty\sqrt{u}\left(\left(\frac{T}{\kappa}\right)^{1/4}\sqrt{2\mu(g^4)}+\sqrt{u}\Vert g^2\Vert_\infty \frac{1}{\sqrt{\kappa}} \right)+2\sqrt{T}\mu(g^2a_{jj})}\right)+4\exp(-u)
	\end{align*}
	where we choose $m_T=\sqrt{T}/(2\sqrt{\kappa})$ in Lemma \ref{lemma: bernstein dirksen form} as in the derivation of \eqref{eq: jump chain 1}.
	Hence for any $u>0$ it holds
	\begin{align*}
		&\P\left(\mathbb{M}^j_T(g)>\sqrt{\Vert a_{jj}\Vert_\infty}\left((512)^{1/4}(\kappa T)^{-1/8}(\sqrt{u}+u)\Vert g\Vert_{L^4(\mu)}+8u(\kappa T)^{-1/4}\Vert g\Vert_\infty+\sqrt{2u}\Vert g\Vert_{L^2(\mu)}\right)\right)
		\\
		&\leq 6\exp(-u),
	\end{align*}
	and arguing as in the derivation of \eqref{eq: conc jump}, we obtain that for large enough $T$ it holds for any $u>0$
	\begin{align*}
		&\P\Bigg(\mathbb{M}^j_T(g)>\sqrt{\Vert a_{jj}\Vert_\infty}\Big(\sqrt{6u}\left((128)^{1/4}(\kappa T)^{-1/8}\Vert g\Vert_{L^4(\mu)}+\Vert g\Vert_{L^2(\mu)}\right)\\&\quad\qquad\qquad+24u\left((\kappa T)^{-1/4}\Vert g\Vert_\infty+(\kappa T)^{-1/8}\Vert g\Vert_{L^4(\mu)}\right)\Big)\Bigg)
		\\
		&\leq 2\exp(-u).
	\end{align*}
	Applying Theorem 3.5 of \cite{dirksen15} now gives for any $p\geq 1$
	\begin{align*}
		&\E\left[\sup_{g\in\mathcal{G}}\vert \mathbb{M}^j_T\vert^p\right]^{1/p}
		\\
		&\leq c\Bigg(\tilde{c}_1\int_0^\infty \log\left(\mathcal{N}(u,\mathcal{G},24\sqrt{\Vert a_{jj}\Vert_{\infty}}((\kappa T)^{-1/4}d_\infty+(\kappa T)^{-1/8}d_{L^4(\mu)})) \right) \d u
		\\
		&\quad +\tilde{c}_2\int_0^\infty\sqrt{\log\left(\mathcal{N}(u,\mathcal{G},\sqrt{6\Vert a_{jj}\Vert_{\infty}} ((128)^{1/4}(\kappa T)^{-1/8}d_{L^4(\mu)}+d_{L^2(\mu))}\right) }\d u\Bigg)
		+2\sup_{g\in\mathcal{G}}\E\left[\vert \mathbb{M}^j_T(g)\vert^p \right]^{1/p},
	\end{align*}
	where we argue analogously to the proof of Proposition \ref{prop: jump moment bound}, which also gives for any $g\in\mathcal{G}$
	\begin{align*}
		&\E\left[\vert \mathbb{M}^j_T(g)\vert^p \right]^{1/p}
		\\
		&\leq 12p\sqrt{\Vert a_{jj}\Vert_\infty}\tilde{c}_1\left((\kappa T)^{-1/4}\Vert g\Vert_\infty+(\kappa T)^{-1/8}\Vert g\Vert_{L^4(\mu)}\right)
		\\&\quad+\sqrt{\frac{3}{2}\Vert a_{jj}\Vert_\infty p}\tilde{c}_2\left((128)^{1/4}(\kappa T)^{-1/8}\Vert g\Vert_{L^4(\mu)}+\Vert g\Vert_{L^2(\mu)}\right),
	\end{align*}
	concluding the proof.
\end{proof}
\begin{proof}[Proof of Lemma \ref{prop: trunc decomp}]
	By linearity it holds
	\begin{align*}
		&\frac{1}{\sqrt{T}}\int_0^Tg(X_{s-})\d X^{j,\delta}_s
		\\
		&=\mathbb{H}_T(gb^j)+\mathbb{M}^j_T(g)+\frac{1}{\sqrt{T}}\int_0^T\int_{B(0,\delta/\gamma_{\min})} g(X_{s-})\gamma^j(X_{s-})z \tilde{N}(\d s,\d z)
		\\
		&\quad+\frac{1}{\sqrt{T}}\int_0^T\int_{B(0,\delta/\gamma_{\min})^{\c}} g(X_{s-})\gamma^j(X_{s-})z N(\d s,\d z)
		\\&\quad- \frac{1}{\sqrt{T}}\int_0^Tg(X_{s-})\int_{B(0,\delta/\gamma_{\min})^{\c}} \gamma^j(X_{s-})z\nu(\d z)\d s
		\\&\quad-\frac{1}{\sqrt{T}}\sum_{0\leq s\leq T} g(X_{s-})\Delta X^j_s\1_{(\delta,\infty)}(\Vert \Delta X_s\Vert)
		\\
		&=\mathbb{H}_T(gb^j)+\mathbb{M}^j_T(g)+\frac{1}{\sqrt{T}}\int_0^T\int_{B(0,\delta/\gamma_{\min})} g(X_{s-})\gamma^j(X_{s-})z \tilde{N}(\d s,\d z)
		\\
		&\quad+\frac{1}{\sqrt{T}}\sum_{0\leq s\leq T} g(X_{s-})\Delta X^j_s \left(\1_{(\delta/\gamma_{\min},\infty)}(\Vert \Delta N_s\Vert)-\1_{(\delta,\infty)}(\Vert \Delta X_s\Vert)\right)
		\\&\quad- \frac{1}{\sqrt{T}}\int_0^Tg(X_{s-})\int_{B(0,\delta/\gamma_{\min})^{\c}} \gamma^j(X_{s-})z\nu(\d z)\d s
		\\
		&\leq \mathbb{H}_T(gb^j)+\mathbb{M}^j_T(g)+\frac{1}{\sqrt{T}}\int_0^T\int_{B(0,\delta/\gamma_{\min})} g(X_{s-})\gamma^j(X_{s-})z \tilde{N}(\d s,\d z)
		\\
		&\quad+\frac{1}{\sqrt{T}}\sum_{0\leq s\leq T} \vert g(X_{s-})\Delta X^j_s\vert \left\vert\1_{(\delta/\gamma_{\min},\infty)}(\Vert \Delta N_s\Vert)-\1_{(\delta,\infty)}(\Vert \Delta X_s\Vert)\right\vert
		\\&\quad- \frac{1}{\sqrt{T}}\int_0^Tg(X_{s-})\int_{B(0,\delta/\gamma_{\min})^{\c}} \gamma^j(X_{s-})z\nu(\d z)\d s
		\\
		&\leq \mathbb{H}_T(gb^j)+\mathbb{M}^j_T(g)+\frac{1}{\sqrt{T}}\int_0^T\int_{B(0,\delta/\gamma_{\min})} g(X_{s-})\gamma^j(X_{s-})z \tilde{N}(\d s,\d z)
		\\
		&\quad+\frac{1}{\sqrt{T}}\sum_{0\leq s\leq T} \vert g(X_{s-})\Delta X^j_s\vert \1_{(\delta/\Vert \gamma\Vert_\infty,\delta/\gamma_{\min}]}(\Vert \Delta N_s\Vert)
		\\&\quad- \frac{1}{\sqrt{T}}\int_0^Tg(X_{s-})\int_{B(0,\delta/\gamma_{\min})^{\c}} \gamma^j(X_{s-})z\nu(\d z)\d s
		\\
		&\leq \mathbb{H}_T(gb^j)+\mathbb{M}^j_T(g)+\frac{1}{\sqrt{T}}\int_0^T\int_{B(0,\delta/\gamma_{\min})} g(X_{s-})\gamma^j(X_{s-})z \tilde{N}(\d s,\d z)
		\\
		&\quad+\frac{\Vert \gamma\Vert_\infty}{\gamma_{\min}\sqrt{T}}\delta\int_0^T\int_{\delta/\Vert \gamma\Vert_\infty<\Vert z\Vert\leq \delta/\gamma_{\min}} \vert g(X_{s-})\vert N(\d s,\d z)
		\\
		&\quad+ \frac{\nu_3\Vert \gamma\Vert_\infty\sigma^2_d}{\delta^2\sqrt{T}}\int_0^T\vert g(X_{s-})\vert \d s
		\\
		&\leq \mathbb{H}_T(gb^j)+\mathbb{M}^j_T(g)+\frac{1}{\sqrt{T}}\int_0^T\int_{B(0,\delta/\gamma_{\min})} g(X_{s-})\gamma^j(X_{s-})z \tilde{N}(\d s,\d z)
		\\
		&\quad+\frac{\Vert \gamma\Vert_\infty}{\gamma_{\min}\sqrt{T}}\delta\int_0^T\int_{\delta/\Vert \gamma\Vert_\infty<\Vert z\Vert\leq \delta/\gamma_{\min}} \vert g(X_{s-})\vert \tilde{N}(\d s,\d z)
		\\
		&\quad+\frac{\Vert \gamma\Vert_\infty}{\gamma_{\min}\sqrt{T}}\delta\int_0^T\vert g(X_{s-})\vert \nu((\delta/\Vert \gamma\Vert_\infty, \delta/\gamma_{\min}])\d s
		\\
		&\quad+ \frac{\nu_3\Vert \gamma\Vert_\infty\sigma^2_d}{\delta^2\sqrt{T}}\int_0^T\vert g(X_{s-})\vert \d s
		\\
		&\leq \mathbb{H}_T(gb^j)+\mathbb{M}^j_T(g)+\frac{1}{\sqrt{T}}\int_0^T\int_{B(0,\delta/\gamma_{\min})} g(X_{s-})\gamma^j(X_{s-})z \tilde{N}(\d s,\d z)
		\\
		&\quad+\frac{\Vert \gamma\Vert_\infty}{\gamma_{\min}\sqrt{T}}\delta\int_0^T\int_{\delta/\Vert \gamma\Vert_\infty<\Vert z\Vert\leq \delta/\gamma_{\min}} \vert g(X_{s-})\vert \tilde{N}(\d s,\d z)+\frac{\Vert \gamma\Vert_\infty^4\nu_3}{\gamma_{\min}\delta^2\sqrt{T}}\int_0^T\vert g(X_{s-})\vert \d s
		\\
		&\quad+ \frac{\nu_3\Vert \gamma\Vert_\infty\sigma^2_d}{\delta^2\sqrt{T}}\int_0^T\vert g(X_{s-})\vert \d s.
	\end{align*}
	The assertion now follows by the Minkowski inequality.
\end{proof}
\begin{proof}[Proof of Proposition \ref{prop: trunc chain}]
	Theorem 2.2 in \cite{applebaum09paper} gives together with Lemma \ref{lemma: qv jump bound comp supp} for any $T,x,y>0$
	\begin{align*}
		&2\exp(-xy)
		\\
		&\geq	\P\Bigg(\vert \mathbb{J}^{j,\delta,1}_T(g)\vert>\frac{x}{\sqrt{T}}\\&\quad\qquad+\frac{1}{\sqrt{T}y} \int_0^T\int_{B(0,\delta/\gamma_{\min})}\exp(yg(X_{s-})\gamma^j(X_{s-})z))-1-yg(X_{s-})\gamma^j(X_{s-})z)\nu(\d z)\d s \Bigg)
		\\
		&\geq	\P\left(\vert\mathbb{J}^{j,\delta,1}_T(g)\vert>\frac{x}{\sqrt{T}}+\frac{y}{\sqrt{T}}\cE(\delta,y)\int_0^Tg(X_{s-})^2\d s  \right),
	\end{align*}
	where
	\[\cE(\delta,y)\coloneq \Vert \gamma\Vert_\infty^2\frac{\nu_2\exp(y\Vert \gamma\Vert_\infty\Vert g\Vert_\infty \delta/\gamma_{\min})}{2}. \]
	Hence, applying Lemma \ref{lemma: bernstein dirksen form} with $m_T=\sqrt{T}/(2\sqrt{\kappa})$ gives for large enough $T$
	\begin{align}
		&\P\left(\vert\mathbb{J}^{j,\delta,1}_T(g)\vert>\frac{x}{\sqrt{T}}+y\cE(\delta,y)32\sqrt{u}\left((T/\kappa)^{1/4}\sqrt{\mu(g^4)}+2\sqrt{u/\kappa}\Vert g^2\Vert_\infty \right)+y\sqrt{T}\cE(\delta,y)\mu(g^2) \right)\notag
		\\
		&\leq 	\P\left(\vert\mathbb{J}^{j,\delta,1}_T(g)\vert>\frac{x}{\sqrt{T}}+\frac{y}{\sqrt{T}}\cE(\delta,y)\int_0^T g(X_s)^2\d s \right)	\notag
		\\&
		\quad+\P\left(\frac{1}{\sqrt{T}}\int_0^T g(X_s)\d s>32\sqrt{u}\left(\sqrt{\Var\left(\frac{1}{\sqrt{\tau}}\int_0^\tau g(X_s)\d s \right)}+\sqrt{u/\kappa}\Vert g\Vert_\infty\right) \right)\notag
		\\
		&\leq2\exp(-xy)+ 2\exp(-u)+2\sqrt{\kappa T}c_\kappa \exp(-\sqrt{\kappa T}/2)\1_{(u,\infty)}\left(\frac{\sqrt{\kappa T}}{8}\right)\notag
		\\
		&\leq2\exp(-xy)+2\exp(-u)+2\exp(-\sqrt{\kappa T}/8)\1_{(u,\infty)}\left(\frac{\kappa\sqrt{T}}{8}\right)\notag
		\\
		&\leq2\exp(-xy)+4\exp(-u)\label{eq: trunc bern}
	\end{align}
	where we again used that for any $\tau>0$
	\[\sqrt{\Var\left(\frac{1}{\sqrt{\tau}}\int_0^\tau g^2(X_s)\d s \right)}\leq \sqrt{\tau\mu(g^4)}. \]
	Now for $u>0,$ setting $x_u=u\left(\sqrt{T\mu(g^2)/u} +T^{3/8}\Vert g\Vert_{L^4(\mu)}+T^{1/4}\Vert g\Vert_\infty\right),y_u=u/x_u$ gives
	\begin{align*}
		&\frac{x_u}{\sqrt{T}}+y_u\cE(\delta,y_u)32\sqrt{u}\left((T/\kappa)^{1/4}\sqrt{\mu(g^4)}+2\sqrt{u/\kappa}\Vert g^2\Vert_\infty \right)+y_u\sqrt{T}\cE(\delta,y_u)\mu(g^2)
		\\
		&\leq \sqrt{u}\left(\Vert g\Vert_{L^2(\mu)} +\sqrt{u}\left(T^{-1/8}\Vert g\Vert_{L^4(\mu)}+T^{-1/4}\Vert g\Vert_\infty\right)\right)
		\\&\quad+y_u\cE(\delta,y_u)32\sqrt{u}\left((T/\kappa)^{1/4}\sqrt{\mu(g^4)}+2\sqrt{u/\kappa}\Vert g^2\Vert_\infty \right)+y_u\sqrt{T}\cE(\delta,y_u)\mu(g^2)
		\\
		&\leq \sqrt{u}\left((1+\cE(\delta,y_u))\Vert g\Vert_{L^2(\mu)}+32\cE(\delta,y_u)T^{-1/8}\kappa^{-1/4}\Vert g\Vert_{L^4(\mu)} \right)
		\\&\quad+u\left(T^{-1/8}\Vert g\Vert_{L^4(\mu)}+T^{-1/4}\left(1+64\cE(\delta,y_u)\kappa^{-1/2} \right)\Vert g\Vert_\infty\right).
	\end{align*}
	Arguing as in the derivation of \eqref{eq: conc jump} this implies together with \eqref{eq: trunc bern} and 
	\[\cE(\delta,y_u)\leq  \Vert \gamma\Vert_\infty^2\frac{\nu_2\exp(\alpha\Vert \gamma\Vert_\infty \sigma^{-1}_d)}{2}\eqcolon \cE_\alpha, \]
	for $\delta\leq \alpha T^{1/4},$
	\begin{align*}
		&\P\Big(\vert\mathbb{J}^{j,\delta,1}_T(g)\vert>\sqrt{3u}\left((1+\cE_\alpha)\Vert g\Vert_{L^2(\mu)}+32\cE_\alpha T^{-1/8}\kappa^{-1/4}\Vert g\Vert_{L^4(\mu)} \right)
		\\&\qquad+3u\left(T^{-1/8}\Vert g\Vert_{L^4(\mu)}+T^{-1/4}\left(1+64\cE_\alpha\kappa^{-1/2} \right)\Vert g\Vert_\infty\right)\Big)\notag
		\\
		&\leq2\exp(-u),
	\end{align*}
	which allows us to apply Theorem 3.5 and Lemma A.2 of \cite{dirksen15}. This gives for any $p\geq 1$
	\begin{align*}
		&\E\left[\sup_{g\in\mathcal{G}}\vert \mathbb{J}^{j,\delta,1}_T\vert^p\right]^{1/p}
		\\
		&\leq C_1\Bigg(\int_0^\infty \log\left(\mathcal{N}\left(u,\mathcal{G},3\left(T^{-1/8}\Vert g\Vert_{L^4(\mu)}+T^{-1/4}\left(1+64\cE_\alpha\kappa^{-1/2} \right)\Vert g\Vert_\infty\right)\right) \right) \d u
		\\
		&\quad +\int_0^\infty\sqrt{\log\left(\mathcal{N}\left(u,\mathcal{G},\sqrt{3}\left((1+\cE_\alpha)\Vert g\Vert_{L^2(\mu)}+32\cE_\alpha T^{-1/8}\kappa^{-1/4}\Vert g\Vert_{L^4(\mu)} \right)\right)\right) }\d u\Bigg)
		\\
		&\quad+2\sup_{g\in\mathcal{G}}\E\left[\vert \mathbb{J}^{j,\delta,1}_T(g)\vert^p \right]^{1/p}
		\\
		&\leq C_1\Bigg(\int_0^\infty \log\left(\mathcal{N}\left(u,\mathcal{G},3\left(T^{-1/8}\Vert g\Vert_{L^4(\mu)}+T^{-1/4}\left(1+64\cE_\alpha\kappa^{-1/2} \right)\Vert g\Vert_\infty\right)\right) \right) \d u
		\\
		&\quad +\int_0^\infty\sqrt{\log\left(\mathcal{N}\left(u,\mathcal{G},\sqrt{3}\left((1+\cE_\alpha)\Vert g\Vert_{L^2(\mu)}+32\cE_\alpha T^{-1/8}\kappa^{-1/4}\Vert g\Vert_{L^4(\mu)} \right)\right)\right) }\d u\Bigg)
		\\
		&\quad+\sup_{g\in\mathcal{G}}\Bigg(3p\tilde{c}_1\left(T^{-1/8}\Vert g\Vert_{L^4(\mu)}+T^{-1/4}\left(1+64\cE_\alpha\kappa^{-1/2} \right)\Vert g\Vert_{\infty}\right)
		\\&\quad+\sqrt{3p}\tilde{c}_2\left((1+\cE_\alpha)\Vert g\Vert_{L^2(\mu)}+32\cE_\alpha T^{-1/8}\kappa^{-1/4}\Vert g\Vert_{L^4(\mu)} \right) \Bigg),
	\end{align*}
	concluding the proof of the first assertion. For the proof of the second assertion, Theorem 2.2 in \cite{applebaum09paper} gives for any $T,x,y>0,$ together with Taylor's theorem
	\begin{align*}
		&2\exp(-xy)
		\\
		&\geq	\P\left(\vert \mathbb{J}^{\delta,2}_T(g)\vert>\frac{x}{\sqrt{T}}+\frac{1}{\sqrt{T}y} \int_0^T\int_{\delta/\Vert \gamma\Vert_\infty<\Vert z\Vert\leq \delta/\gamma_{\min}}\exp(yg(X_{s-}))-1-yg(X_{s-})\nu(\d z)\d s \right)
		\\
		&\geq	\P\left(\vert\mathbb{J}^{\delta,2}_T(g)\vert>\frac{x}{\sqrt{T}}+\frac{y}{2\sqrt{T}}\exp(y\Vert g\Vert_\infty)\nu(B(0,\delta/\Vert \gamma\Vert_\infty)^{\c})\int_0^Tg(X_{s-})^2\d s  \right)
		\\
		&\geq	\P\left(\vert\mathbb{J}^{\delta,2}_T(g)\vert>\frac{x}{\sqrt{T}}+\frac{\nu_2\Vert \gamma\Vert_\infty^2 y}{2\delta^2\sqrt{T}}\exp(y\Vert g\Vert_\infty)\int_0^Tg(X_{s-})^2\d s  \right).
	\end{align*}
	Hence arguing as above, Lemma \ref{lemma: bernstein dirksen form} gives with $m_T=\sqrt{T}(2\sqrt{\kappa})$ that for large enough $T$ it holds for any $u>0$
	\begin{align*}
		&\P\Big(\vert\mathbb{J}^{\delta,2}_T(g)\vert>\frac{x}{\sqrt{T}}+16\frac{\nu_2\Vert \gamma\Vert_\infty^2 y}{\delta^2}\exp(y\Vert g\Vert_\infty)\sqrt{u}\left((T/\kappa)^{1/4}\sqrt{\mu(g^4)}+2\sqrt{u/\kappa}\Vert g^2\Vert_\infty \right)
		\\&\qquad +\sqrt{T}\frac{\nu_2\Vert \gamma\Vert_\infty^2 y}{2\delta^2}\exp(y\Vert g\Vert_\infty)\mu(g^2) \Big)
		\\
		&\leq2\exp(-xy)+ 4\exp(-u),
	\end{align*}
	and choosing $x_u=u\delta^{-1}\left(\sqrt{T\mu(g^2)/u} +T^{3/8}\Vert g\Vert_{L^4(\mu)}+T^{1/4}\Vert g\Vert_\infty\right),y_u=u/x_u$ gives for large enough $T$
	\begin{align*}
		&\frac{x_u}{\sqrt{T}}+16\frac{\nu_2\Vert \gamma\Vert_\infty^2 y_u}{\delta^2}\exp(y_u\Vert g\Vert_\infty)\sqrt{u}\left((T/\kappa)^{1/4}\sqrt{\mu(g^4)}+2\sqrt{u/\kappa}\Vert g^2\Vert_\infty \right)
		\\&\quad+\frac{\nu_2\Vert \gamma\Vert_\infty^2 y_u}{2\delta^2\sqrt{T}}\exp(y_u\Vert g\Vert_\infty)\mu(g^2) 
		\\
		&\leq\delta^{-1}\Vert g\Vert_{L^2(\mu)}\sqrt{u} +u\delta^{-1}\left(T^{-1/8}\Vert g\Vert_{L^4(\mu)}+T^{-1/4}\Vert g\Vert_\infty\right)
		\\&\quad+16\exp(\alpha)\frac{\nu_2\Vert \gamma\Vert_\infty^2 y_u}{\delta^2}\sqrt{u}\left((T/\kappa)^{1/4}\sqrt{\mu(g^4)}+2\sqrt{u/\kappa}\Vert g^2\Vert_\infty \right)
		+\sqrt{T}\frac{\nu_2\Vert \gamma\Vert_\infty^2 y_u}{2\delta^2}\exp(\alpha)\mu(g^2) 
		\\
		&\leq\frac{\sqrt{u}}{\delta}\left(\left(1+\frac{\nu_2\Vert \gamma\Vert_\infty^2}{2}\exp(\alpha)\right)\Vert g\Vert_{L^2(\mu)}+16\exp(\alpha)\nu_2\Vert \gamma\Vert_\infty^2\kappa^{-1/4}T^{-1/8}\Vert g\Vert_{L^4(\mu)}\right) \\&\quad+\frac{u}{\delta}\left(T^{-1/8}\Vert g\Vert_{L^4(\mu)}+T^{-1/4}\Vert g\Vert_\infty\left(1+32\exp(\alpha)\nu_2\Vert \gamma\Vert_\infty^2\kappa^{-1/2}\right) \right).
	\end{align*}
	Now following the same steps as in the proof of the first assertion concludes the proof.
\end{proof}
\section{Proofs for Section \ref{sec: drift est}.}
\begin{proof}[Proof for Theorem \ref{thm: general rate exp}]
	For notational convenience, we suppress the dependence on $j$ throughout this proof. Denoting 
	\[d^2_{\mathbb{G},\tau}(f_1,f_2)=\Vert f_1-f_2\Vert_{\mathbb{G},\tau}^2= \operatorname{Var}\left(\frac{1}{\sqrt{\tau}}\int_0^{\tau} (f_1-f_2)(X_s)\d s\right),\quad f_1,f_2\in L^2(\mu), \] and $\Vert b\Vert_{D,\infty}=\sup_{x: d(x,D)\leq1}\vert b(x)\vert,$ Theorem 3.1 in \cite{dexheimeraihp} gives by choosing $m_T=(p/\kappa)\log(T),$ that there exists $\tau\in [m_t,2m_T],$ such that for large enough $T$
	\begin{align*}
		&\E\left[\sup_{g\in\mathcal{G}_{\h}} \vert \mathbb{H}_T(gb)-\sqrt{T}\mu(gb)\vert^p\right]^{1/p}
		\\
		&\leq C_1\left(\int_0^\infty \log\left(\mathcal{N}\left(u,\mathcal{G}_{\h}b,\frac{2p\log(T)}{\kappa\sqrt{T}}d_\infty\right)\d u\right)+\int_0^\infty \sqrt{\log\left(\mathcal{N}\left(u,\mathcal{G}_{\h}b,d_{\mathbb{G},\tau }\right)\right)}  \right)\d u
		\\&\quad +4\sup_{g\in\mathcal{G}_{\h}}\left(\tilde{c}_1\frac{2p}{\kappa\sqrt{T}}\Vert gb\Vert_\infty +\tilde{c}_2\sqrt{p}\Vert gb\Vert_{\mathbb{G},\tau}+\frac{1}{2\sqrt{T}}\Vert gb\Vert_\infty  \right)
		\\
		&\leq C_1\left(\int_0^\infty \log\left(\mathcal{N}\left(u,\mathcal{G}_{\h}b,\frac{2p\log(T)}{\kappa\sqrt{T}}d_\infty\right) \right)\d u+\int_0^\infty \sqrt{\log\left(\mathcal{N}\left(u,\mathcal{G}_{\h}b,c V(\h)\psi_d(V(\h))d_\infty\right)\right)} \d u \right)
		\\&\quad +4\sup_{g\in\mathcal{G}_{\h}}\left(\tilde{c}_1\frac{2p}{\kappa\sqrt{T}}\Vert gb\Vert_\infty +c\sqrt{ p}\Vert gb\Vert_\infty  V(\h)\psi_d(V(\h))+\frac{1}{2\sqrt{T}}\Vert gb\Vert_\infty  \right)
		\\
		&\leq c\Bigg(\frac{p\log(T)}{\sqrt{T}}\int_0^{ 2\Vert b\Vert_{D,\infty} \Vert K\Vert_\infty}\log\left(\mathcal{N}\left(u,\mathcal{G}_{\h}b,d_\infty\right) \right)\d u\\&\quad+  V(\h)\psi_d(V(\h))\int_0^{2\Vert b\Vert_{D,\infty} \Vert K\Vert_\infty} \sqrt{\log\left(\mathcal{N}\left(u,\mathcal{G}_{\h}b,d_\infty\right)\right)} \d u \Bigg)
		+c\sqrt{\log(T)} V(\h)\psi_d(V(\h)),
	\end{align*}
	where we used Proposition 2.5 in \cite{dexheimeraihp} for bounding the variance distance with
	\begin{equation}\label{def: psi}
	\psi_d(x)\coloneq \begin{cases}
	\sqrt{1-\log(x)},& d\leq 2,\\
	x^{1/d-1/2}, & d\geq 3.
\end{cases} 
	\end{equation}
 For the bound of the variance distance in the case $d=1$ note that by slightly adapting the proof for $d=2$ in Proposition 2.5 of \cite{dexheimeraihp} the same result as in $d=2$ can be achieved. 
	As the proof of Lemma D.2 in \cite{aeckdriftsupp} gives for any $u>0$
	\begin{align}\label{eq: covering number bound}
		\mathcal{N}(u,\mathcal{G}_{\h}b,d_\infty)\leq \left(\frac{2L\Vert b\Vert_{D,\infty}\operatorname{diam}(D)\sum_{i=1}^{d}h_i^{-1} }{u} \right)^d,
	\end{align}
	where $L$ is the Lipschitz constant of $K,$ we get for large enough $T$
	\begin{align}
		&\E\left[\sup_{g\in\mathcal{G}_{\h}} \vert \mathbb{H}_T(gb)-\sqrt{T}\mu(gb)\vert^p\right]^{1/p}\notag
		\\
		&\leq c\Bigg(\frac{\log(T)}{\sqrt{T}} \log\left(\sum_{i=1}^{d}h_i^{-1} \right)^2 + V(\h)\psi_d(V(\h))\sqrt{\log\left(\sum_{i=1}^{d}h_i^{-1} \right) }+\sqrt{\log(T) }\notag V(\h)\psi_d(V(\h)) \Bigg)
		\\
		&\leq V(\h)^{1/2}, \label{eq: moment bound dens}
	\end{align}
	where we used the inequality 
	\begin{equation}\label{eq: sqrt entropy ineq}
		\int_0^C\sqrt{\log(M/u)}\d u\leq 4 C\sqrt{\log(M/C)},\quad\mathrm{if}\quad \log(M/C)\geq 2
	\end{equation} (see p. 592 of \cite{gine2015}).
	Furthermore, Proposition \ref{prop: cont moment bound} gives together with \eqref{eq: covering number bound} for large enough $T$
	\begin{align*}
		&\E\left[\sup_{g\in\mathcal{G}_{\h}}\vert \mathbb{M}_T(g)\vert^p\right]^{1/p}
		\\
		&\leq C_1\Bigg(\int_0^\infty \log\left(\mathcal{N}(u,\mathcal{G}_{\h},24\sqrt{\Vert a\Vert_{\infty}}((\kappa T)^{-1/4}d_\infty+(\kappa T)^{-1/8}d_{L^4(\mu)})) \right) \d u
		\\
		&\quad +\int_0^\infty\sqrt{\log\left(\mathcal{N}(u,\mathcal{G}_{\h},\sqrt{6\Vert a\Vert_{\infty}} (128^{1/4}(\kappa T)^{-1/8}d_{L^4(\mu)}+d_{L^2(\mu)}))\right) }\d u\Bigg)
		\\&\quad
		+ 24p\sqrt{\Vert a_{jj}\Vert_\infty}\tilde{c}_1\sup_{g\in\mathcal{G}_{\h}}\left((\kappa T)^{-1/4}\Vert g\Vert_\infty+(\kappa T)^{-1/8}\Vert g\Vert_{L^4(\mu)}\right)
		\\&\quad+\sqrt{6\Vert a_{jj}\Vert_\infty p}\tilde{c}_2\sup_{g\in\mathcal{G}_{\h}}\left((128)^{1/4}(\kappa T)^{-1/8}\Vert g\Vert_{L^4(\mu)}+\Vert g\Vert_{L^2(\mu)}\right),
		\\
		&\leq C_1\Bigg(24\sqrt{\Vert a\Vert_{\infty}}((\kappa T)^{-1/4}+(\kappa T)^{-1/8}(\Vert \rho\Vert_{\infty}V(\h))^{1/4})\int_0^{2\Vert K\Vert_\infty} \log\left(\mathcal{N}(u,\mathcal{G}_{\h},d_\infty) \right) \d u
		\\
		&\quad +\sqrt{6\Vert a\Vert_{\infty}} (128^{1/4}(\kappa T)^{-1/8}(\Vert \rho\Vert_{\infty}V(\h))^{1/4}+(\Vert \rho\Vert_{\infty}V(\h))^{1/2})\int_0^{2\Vert K\Vert_\infty}\sqrt{\log\left(\mathcal{N}(u,\mathcal{G}_{\h},d_\infty)\right) }\d u\Bigg)
		\\&\quad
		+ 24p\sqrt{\Vert a_{jj}\Vert_\infty}\tilde{c}_1\Vert K\Vert_\infty\left((\kappa T)^{-1/4}+(\kappa T)^{-1/8}(\Vert \rho\Vert_{\infty}V(\h))^{1/4}\right)
		\\&\quad+\sqrt{6\Vert a_{jj}\Vert_\infty p}\tilde{c}_2\Vert K\Vert_\infty\left((128)^{1/4}(\kappa T)^{-1/8}(\Vert \rho\Vert_{\infty}V(\h))^{1/4}+(\Vert \rho\Vert_{\infty}V(\h))^{1/2}\right)
		\\
		&\leq C_1\Bigg(d\sqrt{\Vert a\Vert_{\infty}\Vert \rho\Vert_{\infty}V(\h)}\log\left(\sum_{i=1}^d h_i^{-1}\right)^{-1/2}\int_0^{2\Vert K\Vert_\infty} \log\left(\frac{2L\sum_{i=1}^{d}h_i^{-1} \operatorname{diam}(D)}{u} \right) \d u
		\\
		&\quad +3\sqrt{d\Vert a\Vert_{\infty}\Vert \rho\Vert_{\infty}V(\h)} \int_0^{2\Vert K\Vert_\infty}\sqrt{\log\left(\frac{2L\sum_{i=1}^{d}h_i^{-1} \operatorname{diam}(D)}{u} \right)}\d u\Bigg)
		\\&\quad
		+4\sqrt{\theta\Vert a_{jj}\Vert_\infty\Vert \rho\Vert_{\infty} }\tilde{c}_2\Vert K\Vert_\infty V(\h)^{1/2}\log\left(\sum_{i=1}^d h_i^{-1}\right)^{1/2}
		\\
		&\leq C_1\Bigg(4d\Vert K\Vert_\infty\sqrt{\Vert a\Vert_{\infty}\Vert \rho\Vert_{\infty}V(\h)}\log\left(\sum_{i=1}^d h_i^{-1}\right)^{1/2} 
		+25\Vert K\Vert_\infty\sqrt{d\Vert a\Vert_{\infty}\Vert \rho\Vert_{\infty}V(\h)}\log\left(\sum_{i=1}^{d}h_i^{-1}  \right)^{1/2}\Bigg)
		\\&\quad+4\sqrt{\theta\Vert a_{jj}\Vert_\infty\Vert \rho\Vert_{\infty} }\tilde{c}_2\Vert K\Vert_\infty V(\h)^{1/2}\log\left(\sum_{i=1}^d h_i^{-1}\right)^{1/2}
		\\
		&\leq 29C_1d\Vert K\Vert_\infty\sqrt{\Vert a\Vert_{\infty}\Vert \rho\Vert_{\infty}V(\h)}\log\left(\sum_{i=1}^d h_i^{-1}\right)^{1/2} 
		+4\sqrt{\theta\Vert a_{jj}\Vert_\infty\Vert \rho\Vert_{\infty} }\tilde{c}_2\Vert K\Vert_\infty V(\h)^{1/2}\log\left(\sum_{i=1}^d h_i^{-1}\right)^{1/2}\stepcounter{equation}\tag{\theequation}\label{eq: moment bound cont}
	\end{align*}
	where we used $V(\h)\geq T^{-1/2}\log(\sum_{i=1}^d h_i^{-1})^4$ and \eqref{eq: sqrt entropy ineq}. Turning our attention to the jump part, we obtain for large enough $T$ by Proposition \ref{prop: jump moment bound}
	\begin{align*}
		&\E\Big[\sup_{g\in\mathcal{G}_{\h}}\vert \mathbb{J}^j_T(g)\vert^p\Big]^{1/p}
		\\ &\leq C_1\Bigg(\int_0^\infty \log\left(\mathcal{N}(u,\mathcal{G},3T^{-1/8}d_{L^4(\mu)}+3(64c_1\kappa^{-1/2}+1)T^{-1/4}d_\infty) \right) \d u
		\\
		&\quad +\int_0^\infty\sqrt{\log\left(\mathcal{N}(u,\mathcal{G},\sqrt{3}(1+c_1)d_{L^2(\mu)}+\sqrt{3072}c_1\kappa^{-1/4}T^{-1/8}d_{L^4(\mu)}) \right) }\d u\Bigg)
		\\
		&\quad+3p\tilde{c}_1\sup_{g\in\mathcal{G}_{\h}}\left(T^{-1/8}\Vert g\Vert_{L^4(\mu)}+(64c_1\kappa^{-1/2}+1)T^{-1/4}\Vert g\Vert_\infty  \right)
		\\&\quad+\sqrt{3p}\tilde{c}_2\sup_{g\in\mathcal{G}_{\h}}\left((1+c_1)\Vert g\Vert_{L^2(\mu)}+32c_1\kappa^{-1/4}T^{-1/8}\Vert g\Vert_{L^4(\mu)}\right)
		\\ &\leq C_1\Bigg((3T^{-1/8}(\Vert \rho\Vert_{\infty}V(\h))^{1/4}+3(64c_1\kappa^{-1/2}+1)T^{-1/4})\int_0^{2\Vert K\Vert_\infty} \log\left(\mathcal{N}(u,\mathcal{G},d_\infty) \right) \d u
		\\
		&\quad +(\sqrt{3}(1+c_1)(\Vert \rho\Vert_{\infty}V(\h))^{1/2}+\sqrt{3072}c_1\kappa^{-1/4}T^{-1/8}(\Vert \rho\Vert_{\infty}V(\h))^{1/4})\int_0^{2\Vert K\vert_\infty}\sqrt{\log\left(\mathcal{N}(u,\mathcal{G},d_\infty \right) }\d u\Bigg)
		\\
		&\quad+3p\tilde{c}_1\Vert K\Vert_\infty\left(T^{-1/8}(\Vert \rho\Vert_{\infty}V(\h))^{1/4}+(64c_1\kappa^{-1/2}+1)T^{-1/4}  \right)
		\\&\quad+\sqrt{3p}\tilde{c}_2\Vert K\Vert_\infty\left((1+c_1)(\Vert \rho\Vert_{\infty}V(\h))^{1/2}+32c_1\kappa^{-1/4}T^{-1/8}(\Vert \rho\Vert_{\infty}V(\h))^{1/4}\right)
		\\ 
		&\leq C_1\Bigg((\Vert \rho\Vert_{\infty}V(\h))^{1/2}\log\left(\sum_{i=1}^d h_i^{-1}\right)^{-1/2}\int_0^{2\Vert K\Vert_\infty} \log\left(\mathcal{N}(u,\mathcal{G},d_\infty) \right) \d u
		\\
		&\quad +2(1+c_1)(\Vert \rho\Vert_{\infty}V(\h))^{1/2}\int_0^{2\Vert K\Vert_\infty}\sqrt{\log\left(\mathcal{N}(u,\mathcal{G},d_\infty \right) }\d u\Bigg)
		\\&\quad+3\sqrt{\Vert \rho\Vert_{\infty}\theta}\tilde{c}_2\Vert K\Vert_\infty(1+c_1)V(\h)^{1/2}\log\left(\sum_{i=1}^d h_i^{-1}\right)^{1/2}
		\\ 
		&\leq C_1\Bigg(4\Vert K\Vert_\infty (\Vert \rho\Vert_{\infty}V(\h))^{1/2}\log\left(\sum_{i=1}^d h_i^{-1}\right)^{1/2}  
		+17\Vert K\Vert_\infty(1+c_1)(\Vert \rho\Vert_{\infty}V(\h))^{1/2}\log\left(\sum_{i=1}^d h_i^{-1}\right)^{1/2}  \Bigg)
		\\&\quad+3\sqrt{\Vert \rho\Vert_{\infty}\theta}\tilde{c}_2\Vert K\Vert_\infty(1+c_1)V(\h)^{1/2}\log\left(\sum_{i=1}^d h_i^{-1}\right)^{1/2}
		\\ 
		&\leq C_1\Vert K\Vert_\infty (\Vert \rho\Vert_{\infty}V(\h))^{1/2}\log\left(\sum_{i=1}^d h_i^{-1}\right)^{1/2}  (4
		+17(1+c_1)  )\notag
		\\
		&\quad
		+3\sqrt{\Vert \rho\Vert_{\infty}\theta}\tilde{c}_2\Vert K\Vert_\infty(1+c_1)V(\h)^{1/2}\log\left(\sum_{i=1}^d h_i^{-1}\right)^{1/2}, \stepcounter{equation}\tag{\theequation}\label{eq: moment bound jump}
	\end{align*}
	where we argued analogously to the derivation of \eqref{eq: moment bound cont}. Combining \eqref{eq: moment bound dens}, \eqref{eq: moment bound cont} and \eqref{eq: moment bound jump} with the Minkowski inequality then shows
	\begin{align}\label{eq: moment adap}
		\E\left[\sup_{g\in\mathcal{G}_{\h}} \vert \mathbb{I}^j_T(g)-\sqrt{T}\mu(gb^j)\vert^p\right]^{1/p}\leq cV(\h)^{1/2}\log\left(\sum_{i=1}^d h_i^{-1}\right)^{1/2},
	\end{align}
	which completes the proof of the first assertion as
	\begin{align*}
		\mathcal{R}^{(p)}_\infty(b^j\rho,\bar{b}^{(j)}_{\h,T};D)& \leq 	\mathcal{B}_{b^j\rho}(\h)+T^{-1/2}V(\h)^{-1}\E\left[\sup_{g\in\mathcal{G}_{\h}} \vert \mathbb{I}^j_T(g)-\sqrt{T}\mu(gb^j)\vert^p\right]^{1/p}.
	\end{align*}
\end{proof}
\begin{proof}[Proof of Theorem \ref{thm: rate trunc est}]
	Firstly, it holds
	\begin{align*}
		\mathcal{R}^{(p)}_\infty(b^j\rho,\bar{b}^{(j),\delta}_{\h,T};D)& \leq 	\mathcal{B}_{b^j\rho}(\h)+T^{-1/2}V(\h)^{-1}\E\left[\sup_{g\in\mathcal{G}_{\h}} \vert \mathbb{I}^{j,\delta}_T(g)-\sqrt{T}\mu(gb^j)\vert^p\right]^{1/p}.
	\end{align*}
	and Lemma \ref{prop: trunc decomp} gives
	\begin{align*}
		&\E\left[\sup_{g\in\mathcal{G}_{\h}} \vert \mathbb{I}^{j,\delta}_T(g)-\sqrt{T}\mu(gb^j)\vert^p\right]^{1/p}
		\\
		&\leq \E\left[\sup_{g\in\mathcal{G}_{\h}} \vert \mathbb{H}_T(gb^j)-\sqrt{T}\mu(gb^j)\vert^p\right]^{1/p}+\E\left[\sup_{g\in\mathcal{G}_{\h}}\vert \mathbb{M}^j_T(g)\vert^p\right]^{1/p}+\E\left[\sup_{g\in\mathcal{G}_{\h}}\vert \mathbb{J}^{j,\delta,1}_T(g)\vert^p\right]^{1/p}
		\\
		&\quad+\frac{\Vert \gamma\Vert_\infty\delta}{\gamma_{\min}}\E\left[\sup_{g\in\mathcal{G}_{\h}}\vert \mathbb{J}^{\delta,2}_T(g)\vert^p\right]^{1/p}+\frac{c_2}{\delta^2}\E\left[\sup_{g\in\mathcal{G}_{\h}} \vert \mathbb{H}_T(\vert g\vert)-\sqrt{T}\mu(\vert g\vert)\vert^p\right]^{1/p}+\frac{c_2\sqrt{T}}{\delta^2}\sup_{g\in\mathcal{G}_{\h}}\mu(\vert g\vert).
	\end{align*}
	Now Proposition \ref{prop: trunc chain} yields for large enough $T$
	\begin{align*}
		&\E\left[\sup_{g\in\mathcal{G}_{\h}}\vert \mathbb{J}^{j,\delta,1}_T\vert^p\right]^{1/p}
		\\
		&\leq C_1\Bigg(\int_0^\infty \log\left(\mathcal{N}\left(u,\mathcal{G}_{\h},3\left(T^{-1/8}d_{L^4(\mu)}+T^{-1/4}\left(1+64\cE_\alpha\kappa^{-1/2} \right)d_\infty\right)\right) \right) \d u
		\\
		&\quad +\int_0^\infty\sqrt{\log\left(\mathcal{N}\left(u,\mathcal{G}_{\h},\sqrt{3}\left((1+\cE_\alpha)d_{L^2(\mu)}+32\cE_\alpha T^{-1/8}\kappa^{-1/4}d_{L^4(\mu)} \right)\right)\right) }\d u\Bigg)
		\\
		&\quad+\sup_{g\in\mathcal{G}_{\h}}\Bigg(3p\tilde{c}_1\left(T^{-1/8}\Vert g\Vert_{L^4(\mu)}+T^{-1/4}\left(1+64\cE_\alpha\kappa^{-1/2} \right)\Vert g\Vert_{\infty}\right)
		\\&\quad+\sqrt{3p}\tilde{c}_2\left((1+\cE_\alpha)\Vert g\Vert_{L^2(\mu)}+32\cE_\alpha T^{-1/8}\kappa^{-1/4}\Vert g\Vert_{L^4(\mu)} \right) \Bigg)
		\\
		&\leq C_1\Bigg(\int_0^\infty \log\left(\mathcal{N}\left(u,\mathcal{G}_{\h},3d_\infty\left(T^{-1/8}(\Vert \rho\Vert_{\infty}V(\h))^{1/4}+T^{-1/4}\left(1+64\cE_\alpha\kappa^{-1/2} \right)\right)\right) \right) \d u
		\\
		&\quad +\int_0^\infty\sqrt{\log\left(\mathcal{N}\left(u,\mathcal{G}_{\h},\sqrt{3}d_\infty\left((1+\cE_\alpha)(\Vert \rho\Vert_{\infty}V(\h))^{1/2}+32\cE_\alpha T^{-1/8}(\Vert \rho\Vert_{\infty}V(\h))^{1/4}\kappa^{-1/4} \right)\right)\right) }\d u\Bigg)
		\\
		&\quad+\sup_{g\in\mathcal{G}_{\h}}\Bigg(3p\tilde{c}_1\left(T^{-1/8}(\Vert \rho\Vert_{\infty}V(\h))^{1/4}\Vert g\Vert_\infty+T^{-1/4}\left(1+64\cE_\alpha\kappa^{-1/2} \right)\Vert g\Vert_{\infty}\right)
		\\&\quad+\sqrt{3p}\tilde{c}_2\left((1+\cE_\alpha)(\Vert \rho\Vert_{\infty}V(\h))^{1/2}\Vert g\Vert_\infty+32\cE_\alpha T^{-1/8}\kappa^{-1/4}(\Vert \rho\Vert_{\infty}V(\h))^{1/4}\Vert g\Vert_{\infty} \right) \Bigg)
		\\
		&\leq C_1\Bigg(4(\Vert \rho\Vert_{\infty}V(\h))^{1/2}\log(\sum_{i=1}^d h_i^{-1})^{-1}\int_0^{2\Vert K\Vert_\infty} \log\left(\mathcal{N}\left(u,\mathcal{G}_{\h},d_\infty\right)\right)  \d u
		\\
		&\quad +3(1+\cE_\alpha)(\Vert \rho\Vert_{\infty}V(\h))^{1/2}\int_0^{2\Vert K\Vert_\infty}\sqrt{\log\left(\mathcal{N}\left(u,\mathcal{G}_{\h},d_\infty\right)\right) }\d u\Bigg)
		\\&\quad+3\sqrt{\Vert \rho\Vert_{\infty}\theta\log\left(\sum_{i=1}^d h_i^{-1}\right)}\tilde{c}_2(1+\cE_\alpha)V(\h)^{1/2}\Vert K\Vert_\infty
		\\
		&\leq C_1\Bigg(5dV(\h)^{1/2}\log(\sum_{i=1}^d h_i^{-1})^{-1}\int_0^{2\Vert K\Vert_\infty} \log\left(\frac{\sum_{i=1}^{d}h_i^{-1} }{u}\right)  \d u
		\\
		&\quad +4\sqrt{d}(1+\cE_\alpha)V(\h)^{1/2}\int_0^{2\Vert K\Vert_\infty}\sqrt{\log\left(\frac{\sum_{i=1}^{d}h_i^{-1}}{u}\right) }\d u\Bigg)
		\\&\quad+3\sqrt{\theta\log\left(\sum_{i=1}^d h_i^{-1}\right)}\tilde{c}_2(1+\cE_\alpha)V(\h)^{1/2}\Vert K\Vert_\infty
		\\
		&\leq 
		V(\h)^{1/2}\sqrt{\Vert \rho\Vert_{\infty}\log\left(\sum_{i=1}^{d}h_i^{-1}\right) }\Vert K\Vert_\infty(1+\cE_\alpha)\left(64C_1 \sqrt{d}+3\sqrt{\theta}\tilde{c}_2\right),
	\end{align*}
	where we argued analogously to the derivation of \eqref{eq: moment bound cont}. Similarly it also follows by Proposition \ref{prop: trunc chain} that for large enough $T$
	\begin{align*}
		&\delta\E\left[\sup_{g\in\mathcal{G}_{\h}}\vert \mathbb{J}^{\delta,2}_T(g)\vert^p\right]^{1/p}
		\\
		&\leq \delta C_1\Bigg(\int_0^\infty \log\left(\mathcal{N}\left(u,\mathcal{G},3\delta^{-1}\left(T^{-1/8}d_{L^4(\mu)}+T^{-1/4}d_\infty\left(1+32\exp(\alpha)\nu_2\Vert \gamma\Vert_\infty^2\kappa^{-1/2}\right) \right)\right) \right) \d u
		\\
		&\quad +\int_0^\infty\sqrt{\log\left(\mathcal{N}\left(u,\mathcal{G}_{\h},\sqrt{3}\delta^{-1}\left(\left(1+\frac{\nu_2\Vert \gamma\Vert_\infty^2}{2}\exp(\alpha)\right)d_{L^2(\mu)}+16\exp(\alpha)\nu_2\Vert \gamma\Vert_\infty^2\kappa^{-1/4}T^{-1/8}d_{L^4(\mu)}\right)\right)\right) }\d u\Bigg)
		\\
		&\quad+\sup_{g\in\mathcal{G}_{\h}}\Bigg(3p\tilde{c}_1\left(T^{-1/8}\Vert g\Vert_{L^4(\mu)}+T^{-1/4}\Vert g\Vert_\infty\left(1+32\exp(\alpha)\nu_2\Vert \gamma\Vert_\infty^2\kappa^{-1/2}\right) \right)
		\\&\quad+\sqrt{3p}\tilde{c}_2\left(\left(1+\frac{\nu_2\Vert \gamma\Vert_\infty^2}{2}\exp(\alpha)\right)\Vert g\Vert_{L^2(\mu)}+16\exp(\alpha)\nu_2\Vert \gamma\Vert_\infty^2\kappa^{-1/4}T^{-1/8}\Vert g\Vert_{L^4(\mu)}\right) \Bigg)
		\\
		&\leq V(\h)^{1/2}\sqrt{\Vert \rho\Vert_{\infty}\log\left(\sum_{i=1}^{d}h_i^{-1}\right) }\Vert K\Vert_\infty\left(1+\frac{\nu_2\Vert \gamma\Vert_\infty^2}{2}\exp(\alpha)\right)\left(64\sqrt{d}C_1+3\sqrt{\theta}\tilde{c}_2\right).
	\end{align*}
	Arguing as in the proof of Theorem \ref{thm: general rate exp}, i.e.\ applying Proposition \ref{prop: cont moment bound} and Theorem 3.1 in \cite{dexheimer2020mixing} now concludes the proof.
\end{proof}
\begin{proof}[Proof of Corollary \ref{cor: rate}]
	As the proofs of \hyperref[cor: a]{a)} and \hyperref[cor: b]{b)} are analogous, we only give the proof for the first assertion.
	Arguing as in the derivation of \eqref{eq: moment bound dens} we obtain that for fixed $\alpha>0$ and large enough $T$ it holds for any $p\in[1,\alpha \log(T)]$
	\begin{align*}
		&\E[\Vert \hat{\rho}_{\h^\rho,T}-\E[\hat{\rho}_{\h^\rho,T}]\Vert_{L^\infty(D)}^p]^{1/p}
		\\
		&\leq cT^{-1/2}V(\h^\rho)^{-1}\left(\log(T)^3T^{-1/2}+\log(T)^{1/2}V(\h^\rho)\psi_d(V(\h^\rho))\right),
	\end{align*}
	where $\psi_d$ is defined in \eqref{def: psi}. Together with the standard bias bound under Hölder smoothness assumptions (see, e.g.\ Proposition 1 in \cite{comte2013}) and the specific choice of $\h^\rho$, this yields for large enough $T,$ that for any $p\in[1,\alpha\log(T)]$
	\begin{align}
		&\E[\Vert \hat{\rho}_{\h^\rho,T}-\rho\Vert_{L^\infty(D)}^p]^{1/p}\notag
		\\
		&\leq c\left(\sum_{i=1}^{d}(h_i^\rho)^{\beta_i}+T^{-1/2}V(\h^\rho)^{-1}\left(\log(T)^3T^{-1/2}+\log(T)^{1/2}V(\h^\rho)\psi_d(V(\h^\rho))\right)\right)\notag
		\\
		&\leq c\Phi_{d,\bbeta}(T),\label{eq: cor 2}
	\end{align}
	where we also used $\bar{\bbeta}>d\land 2$.
	Thus for large enough $T,$ Markov's inequality entails for the event $B_T\coloneq\{\Vert \hat{\rho}_{\h^\rho,T}-\rho\Vert_{L^\infty(D)}\leq r(T) \}$ that for any $p\in [1,\alpha\log(T)]$
	\[\P(B_T^\mathsf{c})\leq c^p \exp(- p\sqrt{\log(T)}).  \]
	Hence we obtain if \ref{ass: exp mom} holds
	\begin{align*}
		&\E\left[\Vert(\hat{b}^{(j)}_{\h^b,\h^\rho,T}-b^j)\rho\Vert_{L^\infty(D)}\1_{B^\mathsf{c}_T}  \right]
		\\
		&\leq \E\left[\Vert\hat{b}^{(j)}_{\h^b,\h^\rho,T}\rho\Vert_{L^\infty(D)}\1_{B^\mathsf{c}_T}  \right]+\E\left[\Vert b^j\rho\Vert_{L^\infty(D)}\1_{B^\mathsf{c}_T}  \right]
		\\
		&\leq \E\left[\Vert\hat{b}^{(j)}_{\h^b,\h^\rho,T}\rho\Vert_{L^\infty(D)}^2 \right]^{1/2} c^{p/2}\exp(- p\sqrt{\log(T)}/2)+c^p \exp(- p\sqrt{\log(T)})
		\\
		&\leq r(T)^{-1} c^{p/2} \exp(- p\sqrt{\log(T)}/2)+c^p \exp(- p\sqrt{\log(T)}),
		\\
		&\leq \sqrt{T} c^{p/2} \exp(- p\sqrt{\log(T)}/2)+c^p \exp(- p\sqrt{\log(T)}),
	\end{align*}
	where we used Theorem \ref{thm: general rate exp} in the second to last step, together with the aforementioned bias bound and that $\bar{\bbeta}>d/2$ entails $\h^b\in\mathsf{H}_T$. Choosing $p=5\sqrt{\log(T)}$ then gives for large enough $T$
	\begin{equation}\label{eq: cor 1}
		\E\left[\Vert(\hat{b}^{(j)}_{\h^b,\h^\rho,T}-b^j)\rho\Vert_{L^\infty(D)}\1_{B^\mathsf{c}_T}  \right]
		\leq cT^{-2}.
	\end{equation}
	Furthermore the triangle inequality implies on the event $B_T$ that $\rho/(\vert\hat{\rho}_{\h^\rho,T}\vert +r_T )\leq 1,$ which together with Theorem \ref{thm: general rate exp} gives by arguing as above
	\begin{align*}
		&\E\left[\Vert(\hat{b}^{(j)}_{\h^b,\h^\rho,T}-b^j)\rho\Vert_{L^\infty(D)}\1_{B_T}  \right]
		\\
		&\leq \E\left[\left\Vert\left(\hat{b}^{(j)}_{\h^b,\h^\rho,T}-\frac{b^j\rho}{\vert\hat{\rho}_{\h^\rho,T}\vert+r(T) }\right)\rho\right\Vert_{L^\infty(D)}\1_{B_T}  \right]+\E\left[\left\Vert\left(\frac{b^j\rho}{\vert\hat{\rho}_{\h^\rho,T}\vert+r(T) }-b^j\right)\rho\right\Vert_{L^\infty(D)}\1_{B_T}  \right]
		\\
		&\leq c\left(\left(\frac{\log(T)}{T}\right)^{\frac{\bar{\bbeta}}{2\bar{\bbeta}+d}}+r_T+\Phi_{d,\bbeta}(T) \right),
	\end{align*}
	where we also used \eqref{eq: cor 2}. Combining this with \eqref{eq: cor 1} completes the proof. 
\end{proof}
\section{Proofs for Section \ref{sec: adap}}
Our proofs for the adaptive estimators rely on the following concentration inequalities.
\begin{lemma}\label{lemma: adap}
\begin{enumerate}
	For fixed $\theta>0$ the following holds true:
\item[a)] Let everything be given as in Theorem \ref{thm: general rate exp}. Then for large enough $T$ it holds for any $\h\in\mathscr{H}_T,j\in[d]$
\begin{align*}
	\Pro\left(\sup_{g\in\mathcal{G}_{\h}} \vert \mathbb{I}^j_T(g)-\sqrt{T}\mu(gb^j)\vert\geq T^{1/2}V(\h)A_{T,1}(\h,\theta) \right)&\leq \left(\sum_{i=1}^{d}h_i^{-1} \right)^{-\theta}.
\end{align*}
\item[b)] Let everything be given as in Theorem \ref{thm: rate trunc est}. Then for large enough $T$ it holds for any $\h\in\mathscr{H}_T,j\in[d]$
\begin{align*}
	\Pro\left(\sup_{g\in\mathcal{G}_{\h}} \vert \mathbb{I}^{j,\delta}_T(g)-\sqrt{T}\mu(gb^j)\vert\geq T^{1/2}V(\h)A_{T,2}(\h,\theta,\alpha) \right)&\leq \left(\sum_{i=1}^{d}h_i^{-1} \right)^{-\theta}.
\end{align*}
\end{enumerate}
\end{lemma}
\begin{proof}
	For fixed $\theta>0,$ set $p^*(\h)=\theta \log\left(\sum_{i=1}^{d}h_i^{-1} \right)$ for $\h\in\mathscr{H}_T$ and note that $\mathscr{H}_T\subset \mathsf{H}_T$. Inspecting the proof of Theorem \ref{thm: general rate exp} then gives that under the given assumptions for large enough $T$ it holds for any $\h\in\mathscr{H}_T$ 
	\begin{align*}
		&\E\left[\sup_{g\in\mathcal{G}_{\h}} \vert \mathbb{I}^j_T(g)-\sqrt{T}\mu(gb^j)\vert^{p^*(\h)}\right]^{1/p^*(\h)}
		\\
		&\leq 2V(\h)^{1/2}\log\left(\sum_{i=1}^d h_i^{-1}\right)^{1/2}\Vert K\Vert_\infty\Bigg(C_1\left(21+29d\sqrt{\Vert a\Vert_{\infty}}
		+17c_1 \right) 
		\\&\quad+\sqrt{\theta}\tilde{c}_2\left(3+4\sqrt{\Vert a_{jj}\Vert_\infty }
		+3c_1\right)\Bigg)
		\\
		&=\e^{-1}T^{1/2}V(\h)A_{T,1}(\h,\theta).
	\end{align*}
	Hence, Markov's inequality entails
	\begin{align*}
		\Pro\left(\sup_{g\in\mathcal{G}_{\h}} \vert \mathbb{I}^j_T(g)-\sqrt{T}\mu(gb^j)\vert\geq T^{1/2}V(\h)A_{T,1}(\h,\theta) \right)&\leq \exp(-p^*(\h))=\left(\sum_{i=1}^{d}h_i^{-1} \right)^{-\theta}.
	\end{align*}
	Similarly, the proof of Theorem \ref{thm: rate trunc est} gives that the imposed assumptions entail for large enough $T$ that for any $\h\in\mathscr{H}_T$ it holds
	\begin{align*}
		&\E\left[\sup_{g\in\mathcal{G}_{\h}} \vert \mathbb{I}^{j,\delta}_T(g)-\sqrt{T}\mu(gb^j)\vert^{p^*(\h)}\right]^{1/p^*(\h)}
		\\
		&\leq 2V(\h)^{1/2}\log\left(\sum_{i=1}^d h_i^{-1}\right)^{1/2}\Vert K\Vert_\infty\Bigg(29C_1d\sqrt{\Vert a\Vert_{\infty}}
		+4\sqrt{\theta\Vert a_{jj}\Vert_\infty }\tilde{c}_2
		\\
		&\quad +(1+\cE_\alpha)\left(64C_1 \sqrt{d}+3\sqrt{\theta}\tilde{c}_2\right)
		\\
		&\quad+\frac{\Vert \gamma\Vert_\infty}{\gamma_{\min}}\left(1+\frac{\nu_2\Vert \gamma\Vert_\infty^2}{2}\exp(\alpha)\right)\left(64\sqrt{d}C_1+3\sqrt{\theta}\tilde{c}_2\right)\Bigg)
		\\
		&=\e^{-1}T^{1/2}V(\h)A_{T,2}(\h,\theta,\alpha),
	\end{align*}
	which concludes the proof through Markov's inequality.
\end{proof}
\begin{proof}[Proof of Proposition \ref{thm: adap}]
	In the following we usually suppress the dependencies on $j\in[d]$. For any $\h\in\mathscr{H}_T$ define for $x\in\R^d$
	\begin{align*}
		s_{\h,T}(x)\coloneq \mu(K_{\h}(x-\cdot)b^j(\cdot)),\quad \cS_{\h,T}(x)\coloneq \bar{b}_{\h,T}(x)-s_{\h,T}(x),
	\end{align*}
	and
	\begin{align*}
		s_{\h,\boldsymbol{\eta},T}(x)\coloneq \mu((K_{\h}\star K_{\boldsymbol{\eta}})(\cdot-x)b^j(\cdot)).
	\end{align*}
	Furthermore set
	\begin{align*}
		\zeta_T&\coloneq \sup_{\h\in \cH_T}\left(\Vert\cS^1_{\h,T}\Vert_{\tilde{D},\infty}-A_{T,1}(\h,d)  \right)_+,
	\end{align*}
	where 
	\[\tilde{D}\coloneq \{x\in\R^d\colon d(x,D)\leq 1 \}. \]
	Now for any $\h\in\mathscr{H}_T$ it holds
	\begin{align}\notag
		\Vert \bar{b}_{\hat{\h}_{1,T},T}-b^j\rho\Vert_{D,\infty}&\leq \Vert\bar{b}_{\hat{\h}_{1,T},T}- \bar{b}_{\hat{\h}_{1,T},\h,T}\Vert_{D,\infty}+\Vert\bar{b}_{\hat{\h}_{1,T},\h,T}- \bar{b}_{\h,T}\Vert_{D,\infty}+\Vert \bar{b}_{\h,T}-b^j\rho\Vert_{D,\infty}
		\\\notag
		&\leq \Upsilon_{T,1}(\h)+A_{T,1}(\h,d)+\Upsilon_{T,1}(\hat{\h}_{1,T})+A_{T,1}(\hat{\h}_{1,T},d)+\Vert \bar{b}_{\h,T}-b^j\rho\Vert_{D,\infty}
		\\\notag
		&\leq 2(\Upsilon_{T,1}(\h)+A_{T,1}(\h,d))+\Vert \bar{b}_{\h,T}-b^j\rho\Vert_{D,\infty}
		\\\notag
		&\leq 2(\Upsilon_{T,1}(\h)+A_{T,1}(\h,d))+\Vert\cS_{\h,T}\Vert_{D,\infty}+\cB_{b^j\rho}(\h)
		\\\label{eq: adap 1}
		&\leq 2(\Upsilon_{T,1}(\h)+A_{T,1}(\h,d))+\zeta_T+A_{T,1}(\h,d)+\cB_{b^j\rho}(\h),
	\end{align}
	where we used $\bar{b}_{\hat{\h}_{1,T},\h,T}=\bar{b}_{\h,\hat{\h}_{1,T},T},$ together with the definition of $\hat{\h}_{1,T}$.
	Furthermore Fubini's Theorem for stochastic integrals (see Theorem 65 in \cite{protter2004}) and the symmetry of $K$ give for any $\h,\boldsymbol{\eta}\in\mathscr{H}_T$
	\begin{align*}
		&\Vert\bar{b}_{\h,\boldsymbol{\eta},T}-s_{\boldsymbol{\eta},T} \Vert_{D,\infty}
		\\
		&\leq \Vert\bar{b}_{\h,\boldsymbol{\eta},T}-s_{\h,\boldsymbol{\eta},T} \Vert_{D,\infty}+\Vert s_{\boldsymbol{\eta},T}-s_{\h,\boldsymbol{\eta},T} \Vert_{D,\infty}
		\\
		&=\sup_{x\in D}\left\vert\int_{\R^d} K_{\boldsymbol{\eta}}(u)\frac{1}{T}\int_0^T  K_{\h}(u-(X_{s-}-x))\d X^j_s\d u-\int_{\R^d}K_{\boldsymbol{\eta}}(u) \int_{\R^d} b^j(z)K_{\h}(u-(z-x))\mu(\d z)\d u \right\vert
		\\&\quad +\sup_{x\in D}\left\vert\int_{\R^d}K_{\boldsymbol{\eta}}(x-z)  b^j(z)\mu(\d z)-\int_{\R^d}K_{\boldsymbol{\eta}}(u) \int_{\R^d} b^j(z)K_{\h}(u-(z-x))\mu(\d z)\d u \right\vert
		\\
		&=\sup_{x\in D}\left\vert\int_{\R^d} K_{\boldsymbol{\eta}}(u)\S^1_{\h,T}(u+x)\d u \right\vert
		+\sup_{x\in D}\left\vert\int_{\R^d}K_{\boldsymbol{\eta}}(u)  ((b^j\rho)(u+x)- ((b^j\rho)\ast K_{\h})(u+x))\d u \right\vert
		\\
		&=\Vert K\Vert_{L^1(\R^d)}(\Vert \cS_{\h,T}\Vert_{\tilde{D},\infty}+\cB_{b^j\rho}(\h)),
	\end{align*}
	which implies for any $\h\in\mathscr{H}_T$
	\begin{align*}
		\Upsilon_{T,1}(\h)&\leq \sup_{\eta\in \cH_T}\left(\Vert K\Vert_{L^1(\R^d)}(\Vert \cS_{\h,T}\Vert_{\tilde{D},\infty}+\cB_{b^j\rho}(\h))+\Vert \cS_{\boldsymbol{\eta},T}\Vert_{D,\infty}-A_{T,1}(\boldsymbol{\eta},\ldots) \right)_+
		\\
		&\leq \Vert K\Vert_{L^1(\R^d)}(\zeta_T+A_{T,1}(\h,d)+\cB_{b^j\rho}(\h))+\zeta_T 
		\\
		&\leq (1\lor\Vert K\Vert_{L^1(\R^d)})(2\zeta_T+A_{T,1}(\h,d)+\cB_{b^j\rho}(\h) ).
	\end{align*}
	Combining this with \eqref{eq: adap 1} gives for any $\h\in\mathscr{H}_T$
	\begin{align*}
		&\E\left[\Vert \bar{b}_{\hat{\h}_{1,T},T}-b^j\rho\Vert_\infty \right]
		\leq (1\lor\Vert K\Vert_{L^1(\R^d)})(3\cB_{b^j\rho}(\h) +5A_{T,1}(\h,d)+5\E[\zeta_T]).
	\end{align*}
	Thus it only remains to bound $\E[\zeta_T]$. For this note first that by Hölder's inequality, Lemma \ref{lemma: adap} and \eqref{eq: moment adap} there exists a constant $c>0$ such that for large enough $T$ it holds for any $\h\in\mathscr{H}_T,$ 
	\begin{align*}
		&\E\left[\left(\Vert\cS_{\h,T}\Vert_{\tilde{D},\infty}-A_{T,1}(\h,d)  \right)_+\right]
		\\
		&\leq \sqrt{\E\left[\left(\Vert\cS_{\h,T}\Vert_{\tilde{D},\infty}-A_{T,1}(\h,d)  \right)^2\right]\Pro\left(\Vert\cS_{\h,T}\Vert_\infty\geq A_{T,1}(\h,d)\right)}
		\\
		&\leq c T^{-1/2} V(\h)^{-1/2} \log\left(\sum_{i=1}^{d}h_i^{-1}\right)^{1/2}\left(\sum_{i=1}^{d}h_i^{-1} \right)^{-\theta/2}
		\\
		&= c T^{-1/2}  \log\left(\sum_{i=1}^{d}h_i^{-1}\right)^{1/2}\left(\frac{\left(\prod_{i=1}^dh_i^{-1}\right)^{1/d}}{\sum_{i=1}^{d}h_i^{-1}} \right)^{d/2}
		\\
		&\leq  cd^{-d/2} T^{-1/2}  \log\left(\sum_{i=1}^{d}h_i^{-1}\right)^{1/2},
	\end{align*}
	where we also used the AM--GM inequality in the last step. Hence,
	\begin{align*}
		\E[\zeta_T]&\leq \sum_{\h\in\cH_T}\E\left[\left(\Vert\cS_{\h,T}\Vert_{\tilde{D},\infty}-A_{T,1}(\h,d)  \right)_+\right]
		\\
		&\leq c\operatorname{card}(\cH_T)\sqrt{\frac{\log(T)}{T}}
		\\
		&\leq c \log(T)^d\sqrt{\frac{\log(T)}{T}},
	\end{align*}
	which concludes the proof of the first assertion. The second assertion follows analogously from the results of Section \ref{sec: drift est} and Lemma \ref{lemma: adap}.
\end{proof}
\begin{proof}[Proof of Theorem \ref{thm: adap final}]
	As the proofs of both assertions are almost identical, we omit the proof of the second statement. Fix $j\in[d]$. Noting that for any $\h\in\mathscr{H}_T,x\in D$
	\begin{align*}
		\vert b^j(x)-\hat{b}^{(j)}_{\mathrm{adap},T}(x)\vert
		&\leq \rho_{\min}^{-1}\vert b^j(x)(\hat{\rho}_{\hat{\h}_\rho,T}(x)\lor \rho_{\min})-\bar{b}^{(j)}_{\hat{\h}^j_1,T}(x) \vert 
		\\
		&\leq \rho_{\min}^{-1}\left(\vert b^j(x)(\hat{\rho}_{\hat{\h}_\rho,T}(x)\lor \rho_{\min})-(b^j\rho)(x) \vert+\vert (b^j\rho)(x)-\bar{b}^{(j)}_{\hat{\h}^j_1,T}(x) \vert\right) 
		\\
		&\leq \rho_{\min}^{-1}\left(\Vert b^j\Vert_{L^\infty(D)}\vert \hat{\rho}_{\hat{\h}_\rho,T}(x)-\rho(x) \vert+\vert (b^j\rho)(x)-\bar{b}^{(j)}_{\hat{\h}^j_1,T}(x) \vert\right), 
	\end{align*}
	completes the proof by applying Proposition \ref{thm: adap}, together with the classical anisotropic bias bound, since the optimal bandwidth choices in Corollary \ref{cor: rate} are contained in $\mathscr{H}_T$ if $\bar{\bbeta}>d/2$.
\end{proof}
\printbibliography
\end{document}